\documentclass{amsart}
\usepackage{mathrsfs}
\usepackage{amssymb,latexsym,amsmath,amsthm,enumitem,mathrsfs}
\usepackage{hyperref}
\usepackage{color}
\usepackage{stmaryrd}
\usepackage{mathrsfs}
\usepackage{lineno}
\usepackage{bbm}
\usepackage{scalerel,stackengine}
\usepackage{tikz}
\usepackage{enumitem}
\usepackage{ esint }% See geometry.pdf to learn the layout options. There are lots.
\usepackage{graphicx}
\usepackage{amssymb}
\newtheorem{theorem}{Theorem}

\newtheorem{lemma}[theorem]{Lemma}
\newtheorem{example}[theorem]{Example}
\newtheorem{remark}[theorem]{Remark}
\newtheorem{proposition}[theorem]{Proposition}

\newtheorem{definition}[theorem]{Definition}

\theoremstyle{definition}

\newcommand{\bdA}{\boldsymbol{A}}
\newcommand{\bdE}{\boldsymbol{E}}

\newcommand{\cC}{\mathcal{C}}
\newcommand{\cD}{\mathcal{D}}

\newcommand{\cH}{\mathcal{H}}

\newcommand{\cL}{\mathcal{L}}
\newcommand{\cM}{\mathcal{M}}
\newcommand{\cN}{\mathcal{N}}

\newcommand{\cS}{\mathcal{S}}
\newcommand{\cT}{\mathcal{T}}
\newcommand{\cU}{\mathcal{U}}
\newcommand{\cV}{\mathcal{V}}
\newcommand{\cW}{\mathcal{W}}

\newcommand{\RR}{\mathbb R}
\newcommand{\Q}{\mathbb Q}
\newcommand{\Z}{\mathbb Z}

\newcommand{\N}{\mathbb N}
\newcommand{\T}{\mathbb T}

\renewcommand{\S}{\mathbb S}

\newcommand{\e}{\varepsilon}

\renewcommand{\d}{\de}

\renewcommand{\P}{\mathbb{P}}

\def\bE{\mathbf E}

\def\set4{\mathcal I}
\def\tup14{(1,2,3,4)}

\newtheorem*{comm*}{Comment}
\newtheorem*{lemma*}{Lemma}

\newcommand{\R}{\mathbb{R}}
\newcommand{\D}{\mathbb D}
\newcommand{\de}{\delta} %%%%%\delta
\newcommand{\ga}{\gamma}

\newcommand{\wt}{\widetilde}
\newcommand{\si}{\sigma}

\newcommand{\en}{\epsilon_{\circ}}

\newcommand{\lf}{\lfloor}
\newcommand{\rf}{\rfloor}
\newcommand{\al}{\alpha}
\newcommand{\dir}{\textup{dir}}
\newcommand{\tx}{\textup}
\newcommand{\cQ}{\mathcal Q}
\newcommand{\BL}{\textup{BL}}

\begin{document}

 \author{Shengwen Gan}
 \address{Department of Mathematics\\
 Massachusetts Institute of Technology\\
 Cambridge, MA 02142-4307, USA}
 \email{shengwen@mit.edu}

\keywords{orthogonal projection, exceptional set estimate, Brascamp-Lieb inequality}
\subjclass[2020]{28A75, 28A78}

\date{}

\title{Exceptional set estimate through Brascamp-Lieb inequality}
\maketitle

\begin{abstract}
Fix integers $1\le k<n$, and numbers $a,s$ satisfying $0<s<\min\{k,a\}$. The problem of exceptional set estimate is to determine  
\[ T(a,s):=\sup_{A\subset \mathbb{R}^n\ \textup{dim}A=a}\textup{dim}(\{ 
V\in G(k,n): \textup{dim}(\pi_V(A))<s \}). \]
In this paper, we prove a new upper bound for $T(a,s)$
by using Brascamp-Lieb inequality. As one of the corollaries, we obtain the estimate
\[ T(a,\frac{k}{n}a)\le k(n-k)-\min\{k,n-k\}, \]
which improves a previous result $T(a,\frac{k}{n}a)\le k(n-k)-1$ of He.

By constructing examples, we can determine the explicit value of $T(a,s)$ for certain $(a,s)$: 

When $k\le \frac{n}{2}$, $\beta\in(0,1]$ and $\ga\in(\beta,\frac{k}{n}(1+\beta)]$, we have \[T(1+\beta,\ga)=k(n-k)-k.\] 

When $k\ge \frac{n}{2}$, $\beta\in(0,1]$ and $\ga\in (\beta, (1-\frac{k}{n})+\frac{k}{n}\beta]$, we have 
\[T(n-1+\beta,k-1+\ga)=k(n-k)-(n-k).\] 
\end{abstract}

\section{Introduction}

Let $G(k,n)$ be the set of $m$-dimensional subspaces in $\R^n$. For $V\in G(k,n)$, define $\pi_V:\RR^n\rightarrow V$ to be the orthogonal projection onto $V$.
Throughout the paper,
we use $\dim X$ to denote the Hausdorff dimension of the set $X$. 

There is a classical result proved by Marstrand \cite{marstrand1954some}, who showed that if $A$ is a Borel set in $\R^2$, then the projection of $A$ onto almost every line through the origin has Hausdorff dimension $\min\{1,\dim A\}$. This was generalized to higher dimensions by Mattila \cite{mattila1975hausdorff}, who showed that if $A$ is a Borel set in $\R^n$, then the projection of $A$ onto almost every $k$-dimensional plane through the origin has Hausdorff dimension $\min\{k,\dim A\}$. 

There is a finer question known as the exceptional set estimate. It is formulated as follows. Fix integers $1\le k<n$. For $a,s$ satisfying $0<s<\min\{k,a\}$, find 
\begin{equation}\label{defT}
    T(a,s):=\sup_{A\subset \mathbb{R}^n\ \textup{dim}A=a}\textup{dim}(\{ 
V\in G(k,n): \textup{dim}(\pi_V(A))<s \}). 
\end{equation} 
To avoid the measurability issue, we assume the set $A$ is Borel throughout the paper. We remark that $T(a,s)$ also depends on another two parameters $k,n$, but since $k,n$ are fixed, we just drop them from the notation for simplicity.

We call the set 
\begin{equation}\label{exset}
    E_s(A):=\{ V\in G(k,n): \textup{dim}(\pi_V(A))<s \}
\end{equation} 
the exceptional set.
Roughly speaking, the exceptional set is the set of directions where the projection is small. We call $T(a,s)$ the dimension of the exceptional set. Of course, we trivially have $T(a,s)\le \dim(G(k,n))=k(n-k)$. Here are some nontrivial results on the upper bound of $T(a,s)$. 

\begin{enumerate}[label=(\roman*)]
    \item  $T(a,s)\le k(n-k)+s-k$.\label{1}
    \item  $T(a,s)\le \max\{k(n-k)+s-a,0\}.$\label{2}
    \item $T(a,\frac{k}{n}a)\le k(n-k)-1$. \label{3}
\end{enumerate}
The estimate \ref{1} was first proved by Kaufman in $\R^2$ (see \cite{kaufman1968hausdorff}), and proved by Mattila in higher dimensions (see \cite{mattila1975hausdorff}).
For the the estimate \ref{2}, it was originally obtained by Falconer (see \cite{falconer1982hausdorff}) in a slightly different formulation. Peres-Schlag also proved the result (see \cite{peres2000smoothness}).
The estimate \ref{3} was obtained by He (see \cite{he2020orthogonal}). We remark that estimate \ref{3} beats estimate \ref{1} and \ref{2} if $k=1, a\in (0,1+\frac{1}{n-1})$ or $k=n-1,a\in (n-1-\frac{1}{n-1},n)$. 

In the special case $n=2,k=1$, there were some work by Bourgain \cite{bourgain2010discretized}, Orponen and Shmerkin \cite{Orponen2023projections}, \cite{orponen2021hausdorff}. Shortly after this paper, this special case was completely resolved by Ren and Wang \cite{ren2023furstenberg}. They proved that $T(a,s)=\max\{0,2s-a\}$ when $n=2,k=1$.

In this paper, we will explore a new type of the exceptional set estimate, and as one of the result, we can improve \ref{3}.

\subsection{Statement of the main result}
We fix two integers $1\le k<n$ throughout this paper.
The first theorem is an improvement over the estimate \ref{3}.

\begin{theorem}\label{thm}
    Suppose $T(a,s)$ is defined in \eqref{defT}. We have
    \begin{equation}
        T(a,\frac{k}{n}a)\le k(n-k)-\min\{k,n-k\}.
    \end{equation}
\end{theorem}

    We would like to compare the estimate of $T(a,\frac{k}{n}a)$ given by Theorem \ref{thm} with that given by \ref{1} and \ref{2}. 
    
    When $k\le n/2$, Theorem \ref{thm} gives $T(a,\frac{k}{n}a)\le k(n-k)-k$. This automatically beats \ref{1}. And it beats \ref{2} in the range $a\in (0, \frac{nk}{n-k})$.

    When $k\ge \frac{n}{2}$, Theorem \ref{thm} gives $T(a,\frac{k}{n}a)\le k(n-k)-(n-k)$. It beats \ref{1} if $a>\frac{(2k-n)n}{k}$, and it beats \ref{2} if $a<n$. Therefore, it beats both \ref{1} and \ref{2} in the range $a\in (\frac{(2k-n)n}{k},n)$.

\bigskip
To prove Theorem \ref{thm}, we will prove a more technical theorem.
To handle the numerology, we need to set up some notation.
% \begin{definition}
%     For integers $0\le k\le n $, we write $d(k,n)=k(n-k)$ which is the dimension of $G(k,n)$.
% \end{definition}

\begin{lemma}\label{lem0}
    Let $0\le m\le n, 1\le k,l\le n$ be integers. Let $L\in G(l,n)$. We have
    \begin{align}
        \nonumber&\dim(\{V\in G(k,n):\dim(\pi_L(V))\le m\})=\\
        \label{defE}&\left\{\begin{array}{lr}
             0 & m< \max\{0,l+k-n\} \\
             k(n-k)-(k-m)(l-m) & \max\{0,l+k-n\}\le m\le \min\{l,k\}\\
             k(n-k) & m>\min\{l,k\}
        \end{array}\right.
    \end{align}
\end{lemma}

\begin{proof}
    Since $\dim(\pi_L(V))\ge \max\{0,l+k-n\}$, we have
    \[ \dim(\{V\in G(k,n):\dim(\pi_L(V))\le m\})=0 \]
    for $m< \max\{0,l+k-n\}$.

    Since $\dim(\pi_L(V))\le \min\{l,k\}$, we have
    \[ \dim(\{V\in G(k,n):\dim(\pi_L(V))\le m\})=k(n-k) \]
    for $m> \min\{l,k\}$.

    It remains to consider the case when $\max\{0,l+k-n\}\le m\le \min\{l,k\}$. First by Lemma \ref{easylem}, we have $\dim(\pi_L(V))=\dim (\pi_V(L))$. We are going to apply Lemma \ref{easylem2} with $W=L$ and the roles of $l$ and $m$ switched. We can verify the conditions of Lemma \ref{easylem2} are satisfied: $n-k\ge l-m$, $m\le k,l$. Therefore,
    \begin{equation}
        \dim(\{V\in G(k,n):\dim(\pi_L(V))\le m\})=d(k-m,n-l)+d(m,n-(k-m)).
    \end{equation}
    Here, $d(k,n)=k(n-k)$.
    We can calculate
    \[ \dim(\{V\in G(k,n):\dim(\pi_L(V))\le m\})=k(n-k)-(k-m)(l-m). \]
\end{proof}

\begin{definition}[Definition of $E(m,l)$]
    We define $E(m,l)$ to be the number in \eqref{defE}. We remark that $E(m,l)$ also depends on $n,k$, but $n,k$ are fixed throughout the paper so we drop them from the notation for simplicity.
\end{definition}

\begin{definition}[Definition of $m(t,l)$]\label{defm}
    For $0\le t< k(n-k)$ and integer $1\le l\le n$, we define $m(t,l)$ to be the smallest non-negative integer so that
\[ E(m(t,l),l)> t. \]
\end{definition}

By \eqref{defE}, we see $0\le m(t,l)\le \min\{l,k\}+1$. We also see that $m(\lfloor t\rfloor,l)=m(t,l)$, since $E(m,l)\in \Z$.

Now, we state our technical theorem.

\begin{theorem}\label{mainthm}
    Fix integers $1\le k<n$. For $0<s<\min\{k,a\}$ and $A\subset \R^n$ with $\dim A=a$, define
    \[ E_s(A):=\{ V\in G(k,n): \dim(\pi_V(A))<s \}. \]
    If $s,t$ satisfy
    \begin{equation}\label{scondition}
        s= \sup_{1\le p\le \infty}\inf_{l\in \Z\cap[0,n]}\bigg(\frac{a-l}{p}+m(t,l)\bigg),
    \end{equation} 
    then
    \[ \dim(E_s(A))\le t. \]
\end{theorem}

\begin{remark}
\rm{
    The relationship between $s$ and $t$
    in \eqref{scondition} has a somewhat weird expression in terms of $\sup$ and $\inf$, because we will apply Brascamp-Lieb inequality during the proof, and the Brascamp-Lieb constant has this form (see \eqref{bl0}).

    It seems hard to find an explicit way to express the right hand side of \eqref{scondition} (without using $\sup$, $\inf$). In this paper, we only calculate a special case $s=\frac{k}{n}a$ which gives Theorem \ref{thm}. It is conceivable that we may obtain some other interesting exceptional set estimates by carefully analyzing \eqref{scondition}.
    }
\end{remark}

In Section \ref{section6}, we will find 
the lower bound of $T(a,s)$ by constructing examples. Together with Theorem \ref{thm}, we will see for certain range of $(a,s)$, the upper bound and lower bound of $T(a,s)$ are the same, which gives the exact value of $T(a,s)$. This is the following theorem.
\begin{theorem}\label{explicitT}
    When $k\le \frac{n}{2}$, $\beta\in(0,1]$ and $\ga\in(\beta,\frac{k}{n}(1+\beta)]$, we have \[T(1+\beta,\ga)=k(n-k)-k.\] 

When $k\ge \frac{n}{2}$, $\beta\in(0,1]$ and $\ga\in (\beta, (1-\frac{k}{n})+\frac{k}{n}\beta]$, we have 
\[T(n-1+\beta,k-1+\ga)=k(n-k)-(n-k).\] 
\end{theorem}

\begin{remark}
    {\rm From Theorem \ref{explicitT}, we see that Theorem \ref{thm} is sharp for $a\in (1,2]$ if $k\le \frac{n}{2}$, and for $a\in (n-1,n]$ if $k\ge\frac{n}{2}$,  }
\end{remark}

Here is the structure of the paper. In Section \ref{section2}, we prove Theorem \ref{thm} using Theorem \ref{mainthm}. In Section \ref{section3}, we give a heuristic proof of Theorem \ref{mainthm}. In Section \ref{section4} and \ref{section5}, we prove Theorem \ref{mainthm}. In Section \ref{section6}, we discuss the lower bound of $T(a,s)$ and prove Theorem \ref{explicitT}.

\bigskip

\noindent {\bf Acknowledgement.}
The author would like to thank Prof. Larry Guth for many helpful discussions. The author would also like to thank referees for useful suggestions.

\section{Proof of Theorem \ref{thm}}\label{section2}
\begin{proof}[Proof of Theorem \ref{thm} using Theorem \ref{mainthm}]
We set 
    \[u=\min\{k,n-k\},\ t=k(n-k)-u,\ s=\sup_{1\le p\le \infty}\inf_{l\in \Z\cap[0,n]}\bigg(\frac{a-l}{p}+m(t,l)\bigg).\]
    Theorem \ref{mainthm} implies that 
       \[ T(a,s)\le t.\]
To prove Theorem \ref{thm}, it suffices to show
\begin{equation}
    \frac{k}{n}a\le \sup_{1\le p\le \infty}\inf_{l\in \Z\cap[0,n]}\bigg(\frac{a-l}{p}+m(t,l)\bigg).
\end{equation}
By choosing $p=\frac{n}{k}$, it is further reduced to
\begin{equation}\label{toshow}
    \frac{k}{n}a\le \inf_{l\in \Z\cap[0,n]}\bigg(\frac{(a-l)k}{n}+m(t,l)\bigg).
\end{equation}
By the definition of $u$, we have
\begin{equation}
    k\in \{u,n-u\},\ n\ge 2u.
\end{equation}

We will find an explicit expression for $m(t,l)$. Recall the definition of $m(t,l)$ in Definition \ref{defm}. $m=m(t,l)$ is the smallest non-negative integer satisfying
\[ k(n-k)-(k-m)(l-m)>t, \]
or equivalently,
\[ m^2-(k+l)m+kl+t-k(n-k)<0. \]
We obtain that
\begin{equation*}
    m=\sup\{0,\lf\frac{k+l-\sqrt{(k-l)^2+4(k(n-k)-t)}}{2}\rf+1\}.
\end{equation*}
Plugging $t=k(n-k)-u$, we get
\begin{equation}
    m=\sup\{0,\lf\frac{k+l-\sqrt{(k-l)^2+4u}}{2}\rf+1\}.
\end{equation}
To prove \eqref{toshow}, we just need to prove for any $0\le l\le n$,
\begin{equation*}
    \frac{k}{n}a\le \frac{(a-l)k}{n}+\lf\frac{k+l-\sqrt{(k-l)^2+4u}}{2}\rf+1.
\end{equation*}
Equivalently,
\begin{equation}\label{toshow2}
    \frac{lk}{n}\le \lf\frac{k+l-\sqrt{(k-l)^2+4u}}{2}\rf+1.
\end{equation}

We discuss two cases.

\noindent
\fbox{Case 1: $0\le l\le k$} We write $l=k-\mu$ with $\mu\ge 0$. Then the right hand side of \eqref{toshow2} is
\[ \lf\frac{-\mu-\sqrt{\mu^2+4u}}{2}\rf+k+1. \]
Let $\al$ be the biggest integer such that  $\al<\frac{\mu+\sqrt{\mu^2+4u}}{2}$, or in other words, $\al=-\lf\frac{-\mu-\sqrt{\mu^2+4u}}{2}\rf-1$.
Then \eqref{toshow2} boils down to
\begin{equation}\label{toshow3}
    (k-\mu)k\le n(k-\al).
\end{equation}
There are two subcases.

When $u \le \mu$, then $\al<\frac{\mu+\sqrt{\mu^2+4u}}{2}\le\frac{2\mu+2}{2}$, so $\al\le\mu$. We see \eqref{toshow3} holds since $n>k$.

When $u\ge \mu+1$, we just need to verify \eqref{toshow3} at two endpoints $k=u,k=n-u$. For $k=u$, \eqref{toshow3} becomes $(u-\mu)u\le n(u-\al)$. Using the bound $n\ge 2u$, we just need to show
$u+\mu \ge 2\al$. Noting $\mu^2+4u\le (u+1)^2$, we have $\al<\frac{\mu+u+1}{2}$, and hence $\al\le \frac{\mu+u}{2}$. For $k=n-u$, \eqref{toshow3} becomes $(n-u-\mu)(n-u)\le n(n-u-\al)$, which is equivalent to $u(u+\mu)\le n(u+\mu-\al)$. Using the bound $n\ge 2u$, it boils down to $u+\mu\ge 2\al$ which has been proved.

\medskip
\noindent
\fbox{Case 2: $k\le l\le n$} We write $l=k+\mu$ with $\mu\ge 0$. Then the right hand side of \eqref{toshow2} is
\[ \lf\frac{\mu-\sqrt{\mu^2+4u}}{2}\rf+k+1. \]
Let $\al$ be the biggest integer such that  $\al<\frac{\sqrt{\mu^2+4u}-\mu}{2}$, or in other words, $\al=-\lf\frac{\mu-\sqrt{\mu^2+4u}}{2}\rf-1$.
Then \eqref{toshow2} boils down to
\begin{equation}\label{toshow4}
    (k+\mu)k\le n(k-\al).
\end{equation}
There are two subcases.

When $u \le \mu$, then $\al<\frac{\sqrt{\mu^2+4u}-\mu}{2}\le\frac{\mu+2-\mu}{2}$, so $\al\le0$. We see \eqref{toshow4} holds since $n\ge l=k+\mu$.

When $u\ge \mu+1$, we just need to verify \eqref{toshow4} at two endpoints $k=u,k=n-u$. For $k=u$, \eqref{toshow4} becomes $(u+\mu)u\le n(u-\al)$. Using the bound $n\ge 2u$, we just need to show
$u-\mu \ge 2\al$. Noting $\mu^2+4u\le (u+1)^2$, we have $\al<\frac{u+1-\mu}{2}$, and hence $\al\le \frac{u-\mu}{2}$. For $k=n-u$, \eqref{toshow4} becomes $(n-u+\mu)(n-u)\le n(n-u-\al)$, which is equivalent to $u(u-\mu)\le n(u-\mu-\al)$. Using the bound $n\ge 2u$, it boils down to $u-\mu\ge 2\al$ which has been proved.

\end{proof}

%%%%%%%%%%%%%%%%%%%%%
\section{A heuristic proof of Theorem \ref{mainthm}}\label{section3}

The main idea in the proof of Theorem \ref{mainthm} is as follows. Suppose there is a set $A$ satisfying the condition in Theorem \ref{mainthm} such that
\[ \dim(E_s(A))=\tau>t, \]
then we will derive a contradiction.

We consider the following heuristic setting. $A$ is a disjoint union of $\de$-balls with cardinality $\de^{-a}$. Our exceptional set $E$ is a $\de$-separated subset of $G(k,n)$ with cardinality $\de^{-\tau}$. We also assume $E$ satisfies the $(\de,\tau)$-condition:
\begin{equation}\label{tfrostman}
    \#(E\cap Q_r)\le (r/\de)^\tau,
\end{equation}
for any $Q_r\subset G(k,n)$ being a ball of radius $r$.
The condition that $\dim(\pi_V(A))<s$ is characterized in the following way. For each $V\in E$, we have a set of $(n-k)$-dimensional $\de$-slabs $\T_V=\{T\}$ orthogonal to $V$. In other words, each $T\in\T_V$ is a rectangle of dimensions $\de\times\dots\times\de\times 1\times \dots\times 1$ (with
$n-k$ many $1$'s), and the longest $n-k$ directions are orthogonal to $V$. We also assume $A\subset \bigcup_{T\in\T_V}T$ and $\#\T_V\le \de^{-s+\e_0}$. Under these conditions, one is able to deduce some relationships among the parameters $n,k,a,s,t$. We consider the integral
\begin{equation}\label{integral}
    \int_A (\sum_{V\in E}\sum_{T\in\T_V}1_T)^p, 
\end{equation}
where $1\le p\le \infty$ is to be determined later. Since $A\subset \bigcup_{T\in\T_V}T$, a straightforward observation gives
\[ \int_A (\sum_{V\in E}\sum_{T\in\T_V}1_T)^p\sim |A| (\#E)^p\sim \de^{n-a-\tau p}. \]
On the other hand, we will also obtain an upper bound for the integral.
Obtaining an upper bound for the integral is more complicated. Let us first take a look at a single $\de$-ball $B_\de\subset A$. By our assumption, for each $V\in E$, there is a $T\in\T_V$ such that $T\cap B_\de\neq\emptyset$. Let $K$ be a scale such that $\de\ll K^{-1}\ll 1$, and we partition $G(k,n)$ into $K^{-1}$-balls $\cQ=\{Q\}$ (here $K^{-1}$-ball means ball of radius $K^{-1}$). Denote the number $J=\lfloor K^\tau/2\rfloor$. The key observation is the following inequality.
\begin{equation}\label{mult}
    \de^{-\tau}=\# E\lesssim K^{O(1)} \sum_{Q_1,\dots,Q_J\in\cQ}\bigg(\prod_{j=1}^J\#(Q_j\cap E)\bigg)^{1/J}.
\end{equation}
Here, $Q_1,\dots, Q_J$ are different. 

We give a proof for \eqref{mult}. Label the balls $Q_1, Q_2,\dots,Q_J,\dots $ in $\cQ$ so that $\#(Q_1\cap E)\ge\#(Q_2\cap E)\ge \dots\ge\#(Q_J\cap E)\ge\dots$. We just need to prove $\#(Q_J\cap E)\gtrsim K^{-O(1)}\#E$. By \eqref{tfrostman}, $\sum_{j=1}^J\#(Q_j\cap E)\le J/(K\de)^\tau\le \#E/2$. Therefore,
\[ \#(Q_J\cap E)\ge \frac{1}{\#\cQ} \sum_{j\ge J+1}\#(Q_j\cap E)\ge \frac{1}{\#\cQ} \#E/2\gtrsim K^{-O(1)}\#E. \]

By \eqref{mult}, we can bound the integral \eqref{integral} by the multilinear form
\begin{equation}
    \int_A (\sum_{V\in E}\sum_{T\in\T_V}1_T)^p\lesssim K^{O(1)} \sum_{Q_1,\dots,Q_J\in\cQ}\int \prod_{j=1}^J\bigg(\sum_{V\in Q_j\cap E}\sum_{T\in\T_V}1_T\bigg)^{p/J}.
\end{equation}
The transversality property of these $J$ patches allows us to obtain a nice upper bound. The main tool we use to handle this $J$-linear integral is the Brascamp-Lieb inequality. In the end, we will check the numerology and see the contradiction if $\tau>t$.

\bigskip

Let us work on a very easy case when $n=2,k=1$. We give a heuristic proof for the estimate
\[ \dim (E_{a/2}(A))=0. \]
We just follow the argument as above. We will choose $p=2$.

Suppose by contradiction that $\dim(E_{a/2}(A))>0$. Then there exists $\e_0>0$ such that $\dim(E_{a/2-\e_0}(A))=\tau>0$. On the one hand,
\[ \int_A (\sum_{V\in E}\sum_{T\in\T_V}1_T)^2\sim \de^{2-a-2\tau }. \]
For the upper bound, we notice that $J=\lfloor K^\tau/2\rfloor\ge 2$ (if $K$ is large), so we could just use the bilinear integral as an upper bound:
\[ \int_A (\sum_{V\in E}\sum_{T\in\T_V}1_T)^2\lesssim \int \bigg(\sum_{V\in Q_1\cap E}\sum_{T\in\T_V}1_T\bigg)\bigg(\sum_{V\in Q_2\cap E}\sum_{T\in\T_V}1_T\bigg)\lesssim \de^2 (\#E \de^{-a/2+\e_0})^2,  \]
where $Q_1, Q_2$ are two arcs of $S^1$ with distance $\sim 1$ from each other. Note that $\#E=\de^{-\tau}$. Combining the two estimates gives $1\lesssim \de^{2\e_0}$ which is a contradiction.

\bigskip
For the proof of general $n$ and $k$, our argument is similar. The full details will be given in the next section. 
%%%%%%%%%%%%%%%%%%%%

\section{The proof of Theorem \ref{mainthm}}\label{section4}

\subsection{Preliminary}
We recall some properties of affine Grassmannians from \cite[Section 1.2]{gan2023hausdorff}.

Let $G(k,n)$ be the set of $k$-dimensional subspaces in $\R^n$. For every $k$-plane $V$, we can uniquely write it as
\[ V=\dir(V)+x_V, \]
where $\dir(V)\in G(k,n)$ and $x_V\in V^\perp$. $\dir(V)$ refers to the direction of $V$, as can be seen that $\dir(V)=\dir(V')\Leftrightarrow V\parallel V'$.

In this paper, we use $A(k,n)$ to denote the set of $k$-planes $V$ such that $x_V\in B^n(0,1/2)$. ($B^n(0,1/2)$ is the ball of radius $1/2$ centered at the origin in $\R^n$.)
\[ A(k,n)=\{V: V\tx{~is~a~}k\tx{~dimensional~plane~}, x_V\in B^n(0,1/2)\}. \]
Usually $A(k,n)$ denotes all the $k$-planes in other references, but for our purpose we only care about those $V$ that lie near the origin.

Next, we discuss the metrics on $G(k,n)$ and $A(k,n)$.
For $V_1, V_2\in G(k,n)$, we define
\[ d(V_1,V_2)=\|\pi_{V_1}-\pi_{V_2}\|. \]
Here, $\pi_{V_1}:\R^n\rightarrow V_1$ is the orthogonal projection. We have another characterization for this metric. Define $\rho(V_1,V_2)$ to be the smallest number $\rho$ such that $B^n(0,1)\cap V_1\subset N_{\rho}(V_2)$. We have the comparability of $d(\cdot,\cdot)$ and $\rho(\cdot,\cdot)$.

\begin{lemma}
    There exists a constant $C>0$ (depending on $k,n$) such that
    \[ \rho(V_1,V_2)\le  d(V_1,V_2)\le  C\rho(V_1,V_2).  \]
\end{lemma}
\begin{proof}
    Suppose $B^n(0,1)\cap V_1\subset N_\rho(V_2)$, then for any $v\in\R^n$, we have
    \[ |\pi_{V_1}(v)-\pi_{V_2}(v)|\lesssim \rho|v|, \]
    which implies $d(V_1,V_2)\lesssim \rho$. On the other hand, if for any $|v|\le 1$ we have
    \[ |\pi_{V_1}(v)-\pi_{V_2}(v)|\le d|v|, \]
    then we obtain that $\pi_{V_1}(v)\subset N_{d}(V_2)$. Letting $v$ ranging over $B^n(0,1)\cap V_1$, we get $B^n(0,1)\cap V_1\subset N_{d}(V_2)$, which means $\rho(V_1,V_2)\le  d$. 
\end{proof}

We can also define the metric on $A(k,n)$ given by
\begin{equation}\label{defdist}
    d(V,V')=d(\dir(V),\dir(V'))+|x_{V}-x_{V'}|. 
\end{equation} 
Here, we still use $d$ to denote the metric on $A(k,n)$ and it will not make any confusion.

Similarly, for $V,V'\in A(k,n)$ we can define $\rho(V,V')$ to be the smallest number $\rho$ such that $B^n(0,1)\cap V\subset N_\rho(V')$. We also have the following lemma. We left the proof to the interested readers. 

\begin{lemma}\label{comparablelem}
    There exists a constant $C>0$ (depending on $k,n$) such that for $V,V'\in A(k,n)$,
    \[ C^{-1} d(V,V')\le \rho(V,V')\le C d(V,V').  \]
\end{lemma}

\begin{definition}\label{Vr}
    For $V\in A(k,n)$ and $0<r<1$, we define 
    \[V_r:=N_r(V)\cap B^n(0,1).  \]
\end{definition}
We see that $V_r$ is morally a slab of dimensions $\underbrace{r\times \dots\times r}_{n-k \textup{~times}}\times \underbrace{1\times \dots\times 1}_{k \tx{~times}}$. We usually call $V_r$ a $k$-dimensional $r$-slab. When $k$ is already clear, we simply call $V_r$ an $r$-slab.
If $W$ is a convex set such that $C^{-1} W\subset V_r\subset C W$, then we also call $W$ an $r$-slab. Here, the constant $C$ will be a fixed large constant.

\begin{definition}
    We say two $r$-slabs $V_r$ and $V'_r$ are comparable if $C^{-1}V_r\subset V'_r\subset C V_r$. We say they are essentially distinct if they are not comparable.
\end{definition}

In this paper, we will also consider the balls and $\de$-neighborhood in $A(k,n)$. Recall that we use $B_r(x)$ to denote the ball in $\R^n$ of radius $r$ centered at $x$. To distinguish the ambient space, we will use letter $Q$ to denote the balls in $A(k,n)$. For $V\in A(k,n)$, we use $Q_r(V)$ to denote the ball in $A(k,n)$ of radius $r$ centered at $V$. More precisely,
\[ Q_r(V_0):=\{V\in A(k,n): d(V,V_0)\le r  \}. \]
For a subset $X\subset A(k,n)$, we use the fancy letter $\cN$ to denote the neighborhood in $A(k,n)$:
\[ \cN_r(X):=\{V\in A(k,n): d(V,X)\le r\}. \]
Here, $d(V,X)=\inf_{V'\in X}d(V,V')$.

\bigskip

Next, we briefly recall how to define the Hausdorff dimension for subsets of a metric space. Let $(M,d)$ be a metric space. For $X\subset M$, we denote the $s$-dimensional Hausdorff measure of $X$ under the metric $d$ to be $\cH^s(X;d)$. 
We see that if $d'$ is another metric on $M$ such that $d(\cdot,\cdot)\sim d'(\cdot,\cdot)$, then $\cH^s(X;d)\sim \cH^s(X;d')$. It make sense to define the Hausdorff dimension of $X$ which is independent of the choice of comparable metrics:
\[ \dim X:= \sup\{s: \cH^s(X;d)>0\}. \]

\subsection{\texorpdfstring{$\d$}{Lg}-discretization}

\begin{definition}\label{dessetsd1}
For a number $\de>0$ and any set $X$ (in a metric space), we use $|X|_\de$ to denote the maximal number of $\de$-separated points in $X$.
\end{definition}

\begin{definition}[$(\de,s)$-set]\label{dessetsd2}
Let $\de,s>0$. For $A\subset \R^n$, we say $A$ is a $(\de,s,C)$\textit{-set} if $A$ is $\de$-separated and satisfies the following estimate:
\[
\# (A\cap B_r(x)) \le C (r/\de)^s
\]
for any $x\in \R^n$ and $r\geq \de$. For the purpose of this paper, the constant $C$ is not important, so we simply say $A$ is a $(\de,s)$-set if it satisfies
\[ \# (A\cap B_r(x)) \lesssim (r/\de)^s. \]
\end{definition}

\begin{remark}
Throughout the rest of this paper, We will use $\#E$ to denote the cardinality of a set $E$ and $\lvert \cdot\rvert$ to denote the measure of a region.
\end{remark}

We state two lemmas:
\begin{lemma}\label{frostmans}
Let $\de,s>0$ and let $B\subset \R^n$ be any set with $\mathcal H^s_\infty(B) =: \kappa >0$. Then, there exists a $(\de,s)$-set $P\subset B$ with $\#P\gtrsim \kappa \de^{-s}$.
\end{lemma}
\begin{proof}
See \cite{fassler2014restricted} Lemma 3.13.
\end{proof}

\begin{lemma}\label{frostmans3} Fix $a>0$.
Let $\nu$ be a probability measure satisfying $\nu(B_r)\le C r^a$ for any $B_r$ being a ball of radius $r$. If $A$ is a set satisfying $\nu(A)\ge \kappa$ ($\kappa>0$), then for any $\de>0$ there exists a subset $F\subset A$ such that $F$ is a $(\de,a)$-set and $\# F\gtrsim \kappa C^{-1} \de^{-a}$.
\end{lemma}
\begin{proof}
By the previous lemma, we just need to show $\cH^a_\infty(A)\gtrsim \kappa C^{-1}$. We just check it by definition. For any covering $\{B\}$ of $A$, we have
$$\kappa\le \sum_B\nu(B)\le C \sum_B r(B)^a. $$
Ranging over all the covering of $A$ and taking infimum, we get $$\kappa C^{-1}\lesssim \cH^a_\infty(A).$$
\end{proof}

A key step in the proof of Theorem \ref{mainthm} is the following estimate.

\begin{proposition}\label{discrete1}
Let $0<\tau\le k(n-k), 0<\si\le k$.
For any $\ga\in(0,\frac{1}{10})$, there exists $\de_0(\ga)>0$ depending on $\ga$ such that for any $\de<\de_0(\ga)$, the following is true. 

Suppose $E\subset G(k,n)$ is a $(\de,\tau)$-set with $\#E\gtrsim |\log\de|^{-2}\de^{-\tau}$. Suppose for any $V\in E$, there is a set of $(n-k)$-dimensional $\de$-slabs $\T_V$. Each $T\in\T_V$ lies in $B^n(0,1)$ which has dimensions
\[\underbrace{\de\times \de \times \dots \times \de}_{k \text{~times}}\times \underbrace{1 \times 1 \times \dots \times 1}_{n-k \text{~times}}
\]
such that the $1\times \dots \times 1$-side is orthogonal to $V$, and
\begin{equation}
    \#\T_V\le \de^{-\si}.
\end{equation}
Let $W=\sqcup B_\de\subset B^n(0,1)$ be a disjoint union of $\de$-balls such that for any $B_\de\subset W$,
\begin{equation}
    \#\bigg\{ V\in E: B_\de\cap \bigcup_{T\in\T_V} T\neq \emptyset  \bigg\}\ge |\log\de|^{-3}\# E.
\end{equation}
For $1\le p<\infty$, set $\lambda_p:=\sup_{l\in\Z\cap [0,n]}(l-p \cdot m(t,l))$ (recalling $m(t,l)$ in Definition \eqref{defm}). We have 
\begin{equation}\label{discreteineq}
    \int_{W}\Big(\sum_{V\in E} \sum_{T\in\T_{V}}1_T  \Big)^p\le C_{\ga,p} \de^{-\ga}\de^n\de^{-\lambda_p}(\de^{-\si-\tau})^p.
\end{equation}

\end{proposition}

\bigskip

We will prove Theorem \ref{mainthm} assuming Proposition \ref{discrete1}, and then come back to the proof of Proposition \ref{discrete1}.
Before starting the proof, we state a very useful lemma. We use the following notation. Fix a dimension $m$. For any $\de=2^{-k}$ ($k\in \mathbb N^+$), let $\cD_\de$ denote the lattice $\de$-cubes in $[0,1]^m$.

\begin{lemma}\label{usefullemma}
Suppose $X\subset [0,1]^m$ with $\dim X< s$. Then for any $\e>0$, there exist dyadic cubes $\cC_{2^{-k}}\subset \cD_{2^{-k}}$ $(k>0)$ so that 
\begin{enumerate}
    \item $X\subset \bigcup_{k>0} \bigcup_{D\in\cC_{2^{-k}}}D, $
    \item $\sum_{k>0}\sum_{D\in\cC_{2^{-k}}}r(D)^s\le \e$,
    \item $\cC_{2^{-k}}$ satisfies the $s$-dimensional condition: For $l<k$ and any $D\in \cD_{2^{-l}}$, we have $\#\{D'\in\cC_{2^{-k}}: D'\subset D\}\le 2^{(k-l)s}$.
\end{enumerate}
\end{lemma}

\begin{proof}
See \cite{GGM} Lemma 2.
\end{proof}

\begin{remark}
{\rm 
Besides $[0,1]^m$, this Lemma also works for other compact metric spaces, for example $\S^m$ and $G(k,m)$, which we will use throughout the rest of the paper.
}
\end{remark}

\begin{proof}[Proof of Theorem \ref{mainthm}]
Fix $p\in [1,\infty)$.
Let $s$ be given by
\[ s=\inf_{l\in \Z\cap[0,n]}\bigg(\frac{a-l}{p}+m(t,l)\bigg). \]
Our goal is to prove 
\begin{equation}\label{alsotrue}
    \dim (E_s(A))\le t.
\end{equation} 
If this is true, then we also obtain \eqref{alsotrue} for 
\[s=\sup_{1\le p\le \infty}\inf_{l\in \Z\cap[0,n]}\bigg(\frac{a-l}{p}+m(t,l)\bigg)\]
by a limiting argument. 

By a standard argument, we may assume $A\subset B^n(0,1)$.
We will do a series of reductions.
Noting that 
\[ E_s(A)=\bigcup_{\si<s,\si\in \Q} E_{\si}(A) \]
which is a countable union, we just need to prove
\begin{equation}\label{just1}
    \dim(E_{\si}(A))\le t
\end{equation}
for any $\si<s$.

We choose $\al<\dim(A),\tau<\dim(E_{\si}(A))$. By Frostman’s lemma
there exists a probability measure $\nu_A$ supported on $A$ satisfying $\nu_A(B_r)\le C_{A,\al} r^\al$ for
any $B_r$ being a ball of radius $r$. There exists a probability measure $\nu_E$ supported on $E_{\si}(A)\subset G(k,n)$ satisfying $\nu_E(Q_r)\le C_{E_{\si}(A),\tau} r^\tau$ for
any $Q_r$ being a ball of radius $r$ in $G(k,n)$.
Since our sets $A, E_{\si}(A)$, and numbers $\al,\tau$ are fixed, we may simply write the Frostman's condition as
\begin{equation}\label{frostA}
    \nu_A(B_r)\lesssim r^\al,\ \ \nu_E(Q_r)\lesssim r^\tau. 
\end{equation}
Here the implicit constants depend on $A, E_\si(A),\al, \tau$.
We will later introduce two scales $0<\de\ll K^{-1}\ll 1 $, and what important here is that the implicit constants never depend on $K^{-1},\de$.

We only need to prove
\[
\tau\le t,
\]
as we can send $\tau\rightarrow \dim(E_{\si}(A))$, which gives \eqref{just1}. If this is not true, then there exists $\e_1>0$ such that 
\begin{equation}\label{taubigt}
    \tau>2\e_1+t.
\end{equation} 
We may make $\e_1$ smaller so that 
\begin{equation}\label{sbigsi}
    s>2\e_1+\si.
\end{equation} 
We will derive a contradiction.

Now we start the $\de$-discretization.
Fix a $V\in E_{\si}(A)$. By definition, we have $\dim (\pi_V(A))<\si$. We also fix a small number $\en$ which we will later send to $0$.
By Lemma \ref{usefullemma}, we can find a covering of $\pi_V(A)$ by dyadic cubes $\D_V=\{D\}$, each of which has side-length $2^{-j}$ for some integer $j>|\log_2\en|$. $\D_V$ also satisfies the following property. Define $\mathbb D_{V,j}:=\{D\in\mathbb D_V: r(D)=2^{-j}\}$.
Then we have:
\begin{equation}\label{rsless2}
    \sum_{D\in\mathbb D_V}r(D)^s<1,
\end{equation}
and for each $j$ and $r$-ball $B_r\subset V$, we have
\begin{equation}\label{structure2}
    \#\{D\in \D_{V,j}: D\subset B_r\}\le \left(\frac{r}{2^{-j}}\right)^\si.
\end{equation}
After finding such a $\D_V$, we can define a set of $\de$-slabs by lifting $\D_V$ using the map $\pi_V$. More precisely, we define the set of slabs $\T_{V,j}:=\{\pi^{-1}_V(D): D\in\D_{V,j}\}\cap B^n(0,2)$, $\T_{V}=\bigcup_j\T_{V,j}$. Each slab in $\T_{V,j}$ roughly has dimensions 
\[
\underbrace{2^{-j}\times 2^{-j} \times \dots \times 2^{-j}}_{k \text{~times}}\times \underbrace{1 \times 1 \times \dots \times 1}_{n-k \text{~times}}
\]
such that the $1\times \dots \times 1$-side is orthogonal to $V$. We will use $T$ to denote the slabs in $\T_{V,j}$.  One easily sees that $A\subset \bigcup_{T\in \T_V}T $. By pigeonholing, there exists an integer $j(V)>|\log_2 \en|$ such that
\begin{equation}\label{pigeon1}
    \nu_A(A\cap(\bigcup_{T\in\T_{V,j(V)}}T ))\ge \frac{1}{10j(V)^2}\nu_A(A)=\frac{1}{10j(V)^2}.
\end{equation} 
For each $j>|\log_2\en|$, define $E_{\si,j}(A):=\{V\in E_{\si}(A): j(V)=j\}$. Then we obtain a partition of $E_{\si}(A)$:
\[
E_\si(A)=\bigsqcup_j E_{\si,j}(A).
\]
By pigeonholing again, there exists $j$ such that
\begin{equation}\label{pigeon2}
    \nu_E(E_{\si,j}(A))\ge \frac{1}{10j^2}. 
\end{equation} 
In the rest of the proof, we fix this $j$. We also set $\de=2^{-j}$. We remark that $\de\le \en$. Therefore, we can make $\de\rightarrow 0$ by letting $\en\rightarrow 0$. 

By Lemma \ref{frostmans3}, there exists a $(\de,\tau)$-set $E\subset E_{\si,j}(A)$ with cardinality
\begin{equation}
    \# E\gtrsim |\log \de|^{-2} \de^{-\tau}.
\end{equation}

Next, we consider the set $S:=\{(x,V)\in A\times E: x\in\bigcup_{T\in\T_{V,j}}T \}$. We also use $\mu$ to denote the counting measure on $E$.
Define the sections of $S$:
$$ S_x=\{V\in E: (x,V)\in S\},\ \ \  S_V:=\{x\in A: (x,V)\in S\}. $$
By \eqref{pigeon1} and Fubini, we have
\begin{equation}\label{pigeon3}
    (\nu_A\times \mu)(S)=\sum_{V\in E} \nu_A(S_V)  \ge \frac{1}{10j^2}\# E.
\end{equation}
This implies
\begin{equation}\label{pigeon4}
    (\nu_A\times \mu)\bigg(\Big\{(x,V)\in S: \mu(S_x)\ge\frac{1}{20j^2}\# E  \Big\}\bigg)\ge \frac{1}{20j^2}\#E,
\end{equation}
since
\begin{equation}
    (\nu_A\times \mu)\bigg(\Big\{(x,V)\in S: \mu(S_x)\le\frac{1}{20j^2}\#E  \Big\}\bigg)\le \frac{1}{20j^2}\#E .
\end{equation} 
By \eqref{pigeon4}, we have
\begin{equation}\label{pigeon5}
    \nu_A\bigg(\Big\{x\in A: \mu(S_x)\ge \frac{1}{20j^2}\#E \Big\}\bigg)\ge \frac{1}{20j^2}. 
\end{equation}

Recall $\de=2^{-j}$. By \eqref{pigeon5} and Lemma \ref{frostmans3}, we can find a $\de$-separated  set $H\subset \{x\in A: \# S_x\ge \frac{1}{20j^2}\#E \}$ with cardinality 
\begin{equation}
    \# H\gtrsim |\log\de|^{-2}\de^{-\al}.
\end{equation}

\bigskip

For simplicity, we will omit the subscript $j$ and use $\T_V$ to denote $\T_{V,j}$ which is a set of $(n-k)$-dimensional $\de$-slabs that are orthogonal to $V$. Also by \eqref{structure2},
\begin{equation}
    \#\T_V\le \de^{-\si}. 
\end{equation}
Now we focus on the estimate of the integral
\[ \int_{N_\de(H)}\Big(\sum_{V\in E} \sum_{T\in\T_{V}}1_T  \Big)^p. \]

The lower bound is easy.
By the definition of $H$, we see that for any $x\in H$, there are $\gtrsim |\log\de|^{-2}\#E$ many slabs from $\cup_{V\in E}\T_{V}$ that intersect $x$. In other words,
\[\sum_{V\in E} \sum_{T\in\T_{V}}1_T(x)\gtrsim |\log\de|^{-2}\#E.\]
Therefore, we get the lower bound
\begin{equation}\label{lowerbound}
    \int_{N_\de(H)}\Big(\sum_{V\in E} \sum_{T\in\T_{V}}1_T  \Big)^p\gtrsim |\log\de|^{-2p}\de^n \#H \# E^p.
\end{equation}

Next, we use Proposition \ref{discrete1} to obtain an upper bound. Noting that $H$ is $\de$-separated, $N_\de(H)$ is a disjoint union of $\de$-balls. We let the $W$ in Proposition \ref{discrete1} to be $N_\de(H)$.
We can also check the other conditions in Proposition \ref{discrete1} are satisfied. Therefore, we obtain
\begin{equation}
    \int_{N_\de(H)}\Big(\sum_{V\in E} \sum_{T\in\T_{V}}1_T  \Big)^p\le C_{\ga,p} \de^{-\ga}\de^n\de^{-\lambda_p}(\de^{-\si-\tau})^p.
\end{equation}

Comparing with the lower bound \eqref{lowerbound}, we obtain
\begin{equation}
    \de^n\de^{-\al}\de^{-\tau p}\lesssim_{\ga,p} \de^{-\gamma}|\log\de|^{O(1)} \de^n \de^{-\lambda_p}(\de^{-\si-\tau})^p,
\end{equation}
where $\lambda_p=\sup_{l\in\Z\cap [0,n]}(l-p \cdot m(t,l)).$ Letting $\de\rightarrow 0$ and then $\ga\rightarrow 0$ gives that
\begin{equation}\label{sigmalb}
    \si \ge \frac{\al-\lambda_p}{p}.
\end{equation}
On the other hand, by our choice of $\si$,
\begin{equation}
\si<s-2\e_1=\inf_{l\in\Z\cap [0,n]}(\frac{a-l}{p}+m(t,l))-2\e_1.    
\end{equation}
Combining with \eqref{sigmalb}, we obtain
\begin{equation}
    \inf_{l\in\Z\cap [0,n]}(\frac{\al-l}{p}+m(t,l))<\inf_{l\in\Z\cap [0,n]}(\frac{a-l}{p}+m(t,l))-2\e_1.
\end{equation}
This gives a contradiction by letting $\al\rightarrow a$.

\end{proof}

\section{Proof of Proposition \ref{discrete1}}\label{section5}
We have  
\begin{equation}\label{ineq1}
    \int_{W}\Big(\sum_{V\in E} \sum_{T\in\T_{V}}1_T  \Big)^p\lesssim \de^n\sum_{B_\de\subset W} \#\{T\in \cup_{V\in E}\T_V: B_\de\cap T\neq \emptyset  \}^p.
\end{equation}
For $B_\de\subset W$,
we define the set
\begin{equation}\label{EB}
    E(B_\de):=\{V\in E: B_\de\cap \bigcap_{T\in\T_V}T\neq \emptyset\}.
\end{equation}
Since for each $B_\de$ and $V\in E$, there are $\lesssim 1$ tubes in $\T_V$ that intersect $B_\de$, we have
\[\#\{T\in \cup_{V\in E}\T_V: B_\de\cap T\neq \emptyset  \}\lesssim \# E(B_\de).\]
By the condition of $W$, we have
\[ \#E(B_\de)\gtrsim |\log\de|^{-3}\de^{-\tau}. \]

We can bound \eqref{ineq1} by
\begin{equation}\label{ineq2}
    \int_{W}\Big(\sum_{V\in E} \sum_{T\in\T_{V}}1_T  \Big)^p\lesssim \de^n\sum_{B_\de\subset W} \#E(B_\de)^p.
\end{equation}

Fix a $B_\de\subset W$. We will do the broad-narrow argument for the set $E(B_\de)$.

Fix a large dyadic number $K\gg 1$ which is to be determined later. (Logically, $K$ is chosen before $\de$ and $\de\ll K^{-1}$.) For integers $r\in \N$, we choose $\cQ_{K^{-r}}=\{Q_{K^{-r}}\}$ to be a set of $K^{-r}$-balls that form a partition of $G(k,n)$. (Since $G(k,n)$ is locally diffeomorphic to $[0,1]^{k(n-k)}$, the $K^{-r}$-cubes in the partition of $[0,1]^{k(n-k)}$ gives the partition of $\cQ_{K^{-r}}$.) For $r<r'$ and $Q_{K^{-r}}\in \cQ_{K^{-r}}$, we denote
\[ \cQ_{K^{-r'}}(Q_{K^{-r}}):=\{Q\in \cQ_{K^{-r'}}: Q\subset Q_{K^{-r}} \}. \]

\bigskip

From now on, we use $M,d$ to denote the integers that satisfy
\begin{equation}\label{MN}
    M\sim \log\log \de^{-1},\ d=\lfloor K^{\tau-\e}\rfloor.
\end{equation} 
We have the following lemma.
\begin{lemma}\label{bnlem}
Fix $0<\e<\frac{1}{100}$.
    Suppose $E\subset G(k,n)$ is a $(\de,\tau)$-set with $\#E\ge |\log\de|^{-3}\de^{-\tau}$. If $K$ is large enough depending on $\e$, and $\de$ is small enough depending on $K$, then there exists an integer $r\le M$ such that the following is true.
    There exists $Q_{K^{-r+1}}\in \cQ_{K^{-r+1}}$ such that
    \begin{align}\label{bnlemineq}
        \#\bigg\{Q_{K^{-r}}\in\cQ_{K^{-r}}(Q_{K^{-r+1}}):\#(E\cap Q_{K^{-r}})\ge \de^{-\tau+\e}\bigg\}\ge d.
    \end{align}
\end{lemma}

To prove Lemma \ref{bnlem}, we just need to prove the following lemma, which has more complicated numerology but is easier for induction.

\begin{lemma}\label{bn}
Fix $0<\e<\frac{1}{100}$.
    Suppose $E\subset G(k,n)$ is a $(\de,\tau)$-set with $\#E\ge |\log\de|^{-3}\de^{-\tau}$. If $K$ is large enough depending on $\e$, then there exists an integer $r\lesssim \log\log \de^{-1} $ such that the following is true.
    There exists $Q_{K^{-r+1}}\in \cQ_{K^{-r+1}}$ such that
    \begin{equation}\label{bnineq}
    \begin{split}
        \#\bigg\{&Q_{K^{-r}}\in\cQ_{K^{-r}}(Q_{K^{-r+1}}):\\
        &\#(E\cap Q_{K^{-r}})\ge K^{-n^4} |\log\de|^{-3} K^{\e (r-1)} 2^{-r+1}(K^{r-1}\de)^{-\tau}\bigg\}\ge K^{\tau-\e}.
    \end{split}
    \end{equation}
\end{lemma}
\begin{proof}[Proof of Lemma \ref{bn}]
    We show that if \eqref{bnineq} fails for $1,\dots, r_0$, then there exists $Q_{K^{-r_0}}\in\cQ_{K^{-r_0}}$ such that
    \begin{equation}\label{induc1}
        \#(E\cap Q_{K^{-r_0}})\ge |\log\de|^{-3} K^{\e r_0} 2^{-r_0}(K^{r_0}\de)^{-\tau}.
    \end{equation}
For $0\le r\le r_0$, we will inductively find $Q_{K^{-r}}\in\cQ_{K^{-r}}$ so that
\begin{equation}\label{induc}
        \#(E\cap Q_{K^{-r}})\ge |\log\de|^{-3} K^{\e r} 2^{-r}(K^{r}\de)^{-\tau}.
    \end{equation}
When $r=0$, noting $\cQ_0= G(k,n)$, \eqref{induc} holds by the assumption. Suppose we have done \eqref{induc} for $r-1$, so there exists $Q_{K^{-r+1}}$ such that
\begin{equation}\label{lowQ}
        \#(E\cap Q_{K^{-r+1}})\ge |\log\de|^{-3} K^{\e (r-1)} 2^{-r+1}(K^{r-1}\de)^{-\tau}.
\end{equation}

Fix this $Q_{K^{-r+1}}$. For simplicity, we define the significant balls at scales $K^{-r}$ contained in $Q_{K^{-r+1}}$ as follows. 
\[\cS_{K^{-r}}:=\]
\[ \bigg\{Q_{K^{-r}}\in\cQ_{K^{-r}}(Q_{-r+1}): \#(E\cap Q_{K^{-r}})\ge K^{-n^4} |\log\de|^{-3} K^{\e (r-1)} 2^{-r+1}(K^{r-1}\de)^{-\tau}\bigg\}. \]
We also let
\[ \cS^c_{K^{-r}}:=\cQ_{K^{-r}}(Q_{K^{-r+1}})\setminus \cS_{K^{-r}}. \]
The reason we call them significant balls is that the cardinality $\#(E\cap Q_{K^{-r}})$ almost reaches the Frostman upper bound: $\#(E\cap Q_{K^{-r}})\lesssim (K^r\de)^{-\tau}$. 

We have
\begin{equation}\label{leftright}
    \#(E\cap Q_{K^{-r+1}})=\sum_{Q_{K^{-r}}\in\cS_{K^{-r}}}\#(E\cap Q_{K^{-r}})+\sum_{Q_{K^{-r}}\in\cS^c_{K^{-r}}}\#(E\cap Q_{K^{-r}})=:I+II.
\end{equation}

Since by the assumption \eqref{bnineq} fails, we have
\begin{equation}\label{I}
    I\le \#\cS_{K^{-1}} \max_{Q_{K^{-r}}} \#(E\cap Q_{K^{-r}})\le K^{\tau-\e}\max_{Q_{K^{-r}}}\#(E\cap Q_{K^{-r}}).
\end{equation}
We also have the estimate for $II$.
\begin{align}\label{II}
    II&\le \#\cQ_{K^{-r}}(Q_{K^{-r+1}})K^{-n^4} |\log\de|^{-2} K^{\e (r-1)} 2^{-r+1}(K^{r-1}\de)^{-\tau}\\
    \nonumber &\le \frac{1}{2}|\log\de|^{-3} K^{\e (r-1)} 2^{-r+1}(K^{r-1}\de)^{-\tau},
\end{align}
if $K>10$.

Plugging \eqref{I}, \eqref{II} into \eqref{leftright} and noting \eqref{lowQ}, we obtain
\begin{equation}
    \frac{1}{2}|\log\de|^{-3} K^{\e (r-1)} 2^{-r+1}(K^{r-1}\de)^{-\tau}\le K^{\tau-\e}\max_{Q_{K^{-r}}}\#(E\cap Q_{K^{-r}}).
\end{equation}
In other words, there exists $Q_{K^{-r}}\in \cQ_{K^{-r}}$ such that
\[ \#(E\cap Q_{K^{-r}})\ge |\log\de|^{-2} K^{\e r} 2^{-r}(K^{r}\de)^{-\tau}.\]
This finishes the proof of \eqref{induc}.

Now, we finish the proof of lemma. If \eqref{bnineq} fails for all $r\lesssim \log\log\de^{-1}$, then by \eqref{induc} and noting $E$ is a $(\de,\tau)$-set, we have
\begin{equation}
    (K^{r}\de)^{-\tau}\gtrsim |\log\de|^{-2} K^{\e r} 2^{-r}(K^{r}\de)^{-\tau},
\end{equation}
for some $r\sim \log\log \de^{-1}$ which gives a contradiction, provided $\frac{K^\e}{2}\gg 1$.
\end{proof}

\bigskip

The following is the broad estimate.

\begin{lemma}\label{broadestthm}
Fix $0<\e<\frac{1}{100}$.
Suppose $E(B_\de)\subset G(k,n)$ is a $(\de,\tau)$-set with $\# E(B_\de)\ge |\log\de|^{-3}\de^{-\tau}$. 
If $K$ is large enough depending on $\e$, and $\de$ is small enough depending on $K$, then we have
    \begin{equation}\label{broadest}
        \#E(B_\de)\lesssim_\e \de^{-\e} \sum_{r=1}^M\sum_{Q_{K^{-r+1}}\in\cQ_{K^{-r+1}}}  \sum_{U_1,\dots,U_d\in \cQ_{K^{-r}}(Q_{K^{-r+1}})}\prod_{i=1}^d \#(E(B_\de)\cap U_i)^{\frac{1}{d}}. 
    \end{equation}
Here in the $3$rd summation, $U_1,\dots,U_d$ are different.
\end{lemma}

\begin{remark}
{\rm
    We explain why it is called the broad estimate. We first observe that the number of terms in the summation on the RHS of \eqref{broadest} is $\le M C_K$ ($C_K$ is some large constant depending on $K$), which is negligible compared to $\de^{-\e}$ if $K$ is a fixed large constant. (Recall $M$ in \eqref{MN}.) Therefore, we may assume one term on the RHS dominates the LHS. In other words, there exist patches $U_1,\dots,U_d$ in $G(k,n)$ such that
    \[ \#E(B_\de)\lessapprox \prod_{i=1}^d \#(E(B_\de)\cap U_i)^{\frac{1}{d}}. \]
It is equivalent to 
\[ \#E(B_\de)\lessapprox  \#(E(B_\de)\cap U_i), \]
for $i=1,\dots,d$.
This shows that $E(B_\de)$ is not concentrated in a single patch $U_i$, but instead has a significant portion in a ``broad" range of patches $\{U_i\}_{i=1}^d$.
}
\end{remark}

\begin{proof}[Proof of Lemma \ref{broadestthm}]
    Apply Lemma \ref{bnlem} to the set $E=E(B_\de)$. There exists an integer $r\le M$, $Q_{K^{-r+1}}\in \cQ_{K^{-r+1}}$, and $U_1,\dots,U_d\in \cQ_{K^{-r}}(Q_{K^{-r+1}})$ such that
    \begin{equation}
        \#(E(B_\de)\cap U_i)\ge \de^{-\tau+\e}
    \end{equation}
    for $i=1,\dots, d$.
Since $E(B_\de)$ is a $(\de,\tau)$-set, we have $\#E(B_\de)\lesssim \de^{-\tau}$.    
Therefore, we obtain
\[ \#E(B_\de)\lesssim \de^{-\e}\prod_{i=1}^d \#(E(B_\de)\cap U_i)^{\frac{1}{d}}, \]
which finishes the proof of \eqref{broadest}.
\end{proof}

Plugging to \eqref{ineq2}, we obtain

\begin{equation}\label{left}
\begin{split}
    &\int_{W}\Big(\sum_{V\in E} \sum_{T\in\T_{V}}1_T  \Big)^p\\
    &\lesssim_\e \de^{-\e}\sum_{r=1}^M\sum_{Q\in\cQ_{K^{-r+1}}}  \sum_{U_1,\dots,U_d\in \cQ_{K^{-r}}(Q)} \sum_{B_\de\subset W}\de^n  \prod_{i=1}^d \#(E(B_\de)\cap U_i)^{\frac{p}{d}}\\
    &\lesssim \de^{-\e}\sum_{r=1}^M\sum_{Q\in\cQ_{K^{-r+1}}}  \sum_{U_1,\dots,U_d\in \cQ_{K^{-r}}(Q)} \int_{B^n(0,1)} \prod_{i=1}^d (\sum_{V\in E\cap U_i}\sum_{T\in\T_V}1_T)^{\frac{p}{d}}.
    \end{split}
\end{equation}
In the first inequality, we use the fact that the number of summands is $\le M C_K$ which can be controlled by $\de^{-\e}$.
In the last inequality, we use the fact that $1_T$ is locally constant on $\de$-balls for $T\in\T_V$.

\bigskip

\subsection{Rescaling} We will discuss the rescaling in this subsection. While the notation may be complicated, the geometry behind it is simple. In this subsection, our notation is: we use letter $V$ for a point in $G(k,n)$, i.e., a $k$-dimensional plane, while letter $U$ for some small ball in $G(k,n)$.

Fix a $Q\in\cQ_{K^{-r+1}}$ and $U_1,\dots,U_d\in\cQ_{K^{-r}}(Q)$.
We will write 
\begin{equation}\label{resc}
    \int_{B^n(0,1)} \prod_{i=1}^d (\sum_{V\in E\cap U_i}\sum_{T\in\T_V}1_T)^{\frac{p}{d}}
\end{equation}
in a rescaled version.

Let $V^\circ\in G(k,n)$ be the center of $Q$. In other words, $Q=Q_{K^{-r+1}}(V^\circ)$. Let $(V^\circ)^\perp_{K^{-r+1}}$ be the $K^{-r+1}$-neighborhood of $(V^\circ)^\perp$ that is truncated in $B^n(0,1)$ (see Definition \ref{Vr}).
Then $(V^\circ)^\perp_{K^{-r+1}}$ is roughly a rectangle of dimensions
\[ \underbrace{K^{-r+1}\times K^{-r+1} \times \dots \times K^{-r+1}}_{k \text{~times}}\times \underbrace{1 \times 1 \times \dots \times 1}_{n-k \text{~times}}. \]
For any $V\in Q$, by Lemma \ref{comparablelem}, we have $V\cap B^n(0,1)\subset C V^\circ_{K^{-r+1}}$. Therefore, any $T\in \T_V$ appeared in \eqref{resc} is contained in a translation of $C (V^\circ)^\perp_{K^{-r+1}}$.

\begin{figure}[ht]
\centering
\begin{minipage}[b]{0.85\linewidth}
\includegraphics[width=11cm]{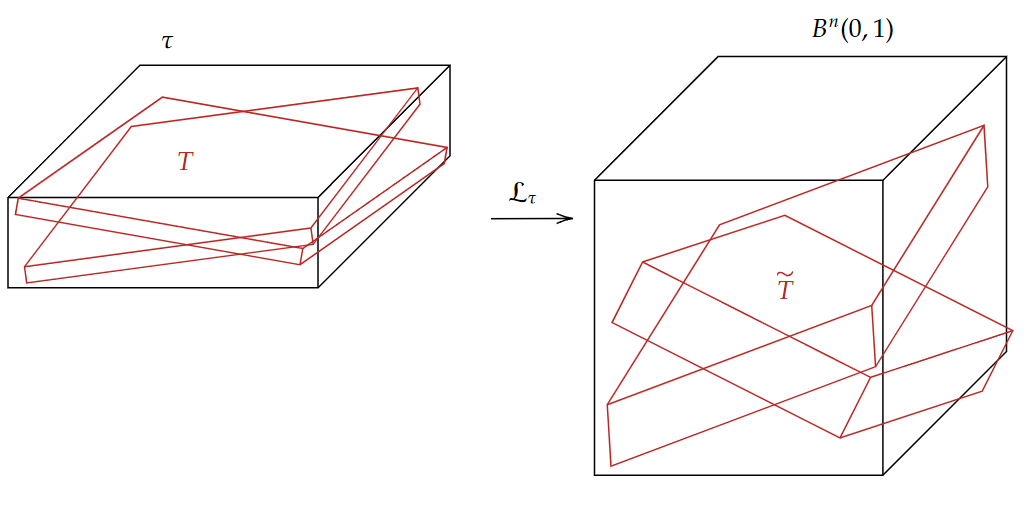}
\caption{Rescaling}
\label{rescale}
\end{minipage}
\end{figure}

To estimate \eqref{resc}, we will first partition the integration domain. We partition $B^n(0,1)$ into slabs denoted by $\cT=\{\tau\}$. Here each $\tau$ is a translation of $ (V^\circ)^\perp_{K^{-r+1}}$, and $\#\cT\sim (K^{r-1})^{k}$.
Since for each fixed $T\in \T_V$, there exists a $\tau\in\cT$ such that $T\subset C\tau$ for some constant $C$. Therefore,
\begin{equation}\label{bdd}
    \#\{ \tau\in\cT: T\cap \tau\neq \emptyset \}\lesssim 1. 
\end{equation} 

We have
\begin{equation}\label{ineq0}
    \int_{B^n(0,1)} \prod_{i=1}^d (\sum_{V\in E\cap U_i}\sum_{T\in\T_V}1_T)^{\frac{p}{d}}=\sum_\tau \int_{\tau} \prod_{i=1}^d (\sum_{V\in E\cap U_i}\sum_{T\in\T_V, T\cap \tau\neq \emptyset}1_T)^{\frac{p}{d}}.
\end{equation}

Next, we will do the rescaling for each $\tau$. We first translate $\tau$ to the origin and do $K^{r-1}$-dilation in the orthogonal directions of $\tau$, so that $\tau$ roughly becomes $B^n(0,1)$. We denote this rescaling map by $\cL_\tau=\cL\circ T_\tau$, where $T_\tau$ is the translation and $\cL$ is the dilation. After the rescaling, we see that each $T$ such that $T\cap \tau\neq\emptyset$ becomes a $\de K^{r-1}\times \dots\times \de K^{r-1}\times 1\times \dots\times 1$-slab contained in $B^n(0,1)$. We define them to be $\wt T$. See Figure \ref{rescale}.
Taking into consideration of the rescaling factor, we have
\begin{equation}\label{ineq22}
   \int_{\tau} \prod_{i=1}^d (\sum_{V\in E\cap U_i}\sum_{T\in\T_V, T\cap\tau\neq\emptyset}1_T)^{\frac{p}{d}}=K^{-(r-1)(n-k)}\int_{B^n(0,1)} \prod_{i=1}^d (\sum_{V\in E\cap U_i}\sum_{T\in\T_V, T\cap \tau\neq\emptyset}1_{\wt T})^{\frac{p}{d}}.
\end{equation} 

We want to rewrite the right hand side. Note that if $T\in \T_V$, then the core of $T$ is orthogonal to $V$, and hence parallel to $V^\perp$. Therefore, after the rescaling, the core of $\wt T$ is parallel to $\cL(V^\perp)$, and hence orthogonal to $(\cL(V^\perp))^\perp$. This motivates us to define the following map
\begin{equation*}
    \begin{split}
        \cL^*: G(k,n)\rightarrow G(k,n)\\
        \cL^*(V)=(\cL(V^\perp))^\perp.
    \end{split}
\end{equation*} 
We see that if $T\in\T_V$, then the core of $\wt T$ is orthogonal to $\cL^*(V)$.

Next, we note that under the map $\cL^*$, the $K^{-r}$-ball $U_i$ in $G(k,n)$ becomes an $O(K^{-1})$-ball which we denote by 
\[\wt U_i:=\cL^*(U_i)=\{\cL^*(V): V\in U_i\}.\] 
Actually, one can verify that if $U_i=Q_{K^{-r}}(V_i)$, then $\wt U_i$ is morally an $O(K^{-1})$-ball centered at $\cL^*(V_i)$. We briefly explain the reason: Since taking orthogonal complement gives an identification between $G(k,n)$ and $G(n-k,n)$, we can view $\{V^\perp: V\in U_i\}$  as a $K^{-r}$-ball in $G(n-k,n)$ centered at $V_i^\perp$; Then by Lemma \ref{comparablelem}, we see that $\{\cL(V^\perp):V\in U_i\}$ is morally an $O(K^{-1})$-ball centered at $\cL(V_i^\perp)$; Finally by taking the orthogonal complement back, we see that $\wt U_i=\{ (\cL(V^\perp))^\perp:V\in U_i \}$ is morally an $O(K^{-1})$-ball centered at $\cL^*(V_i)$.

For $\wt V\in\wt U_i$, there is a $V\in U_i$ such that $\cL^*(V)=\wt V$. Define
\begin{equation}
    \wt\T_{\wt V,\tau}:=\{ \wt T: T\in\T_{V}, T\cap \tau\neq\emptyset \}.
\end{equation}
By \eqref{bdd}, we have
\begin{equation}\label{sumTV}
    \sum_\tau \#\wt\T_{\wt V,\tau}\lesssim \#\T_V.
\end{equation}
Now we can rewrite \eqref{ineq22} as
\begin{equation}\label{ineq3}
    \int_{\tau} \prod_{i=1}^d (\sum_{V\in E\cap U_i}\sum_{T\in\T_V, T\cap\tau\neq\emptyset}1_T)^{\frac{p}{d}}=K^{-(r-1)(n-k)}\int_{B^n(0,1)} \prod_{i=1}^d (\sum_{\wt V\in \cL^*(E\cap U_i)}\sum_{\wt T\in\wt\T_{\wt V,\tau}}1_{\wt T})^{\frac{p}{d}}.
\end{equation}
The reason we do the rescaling is to use Brascamp-Lieb inequality at scale $K^{-1}$, as can be seen on the right hand side of \eqref{ineq3} that $\cL^*(E\cap U_i)$ is contained in a $O(K^{-1})$-ball in $G(k,n)$. We will talk about it in the next subsection.

\bigskip

\subsection{Brascamp-Lieb inequality}
The version of Brascamp-Lieb inequality we are going to use is due to Maldague. See Theorem 2 in \cite{maldague2022regularized}.

\begin{theorem}[Maldague]\label{Domthm}
    Let $W_j\in G(k,n)$ for $j=1,\dots,J$. Fix $p\in[1,J]$. Define
    \begin{equation}\label{bl0}
        \BL(\{W_j\}_{j=1}^J,p):=\sup_{L\leq \R^n }(\dim L-\frac{p}{J}\sum_{j=1}^J\dim\pi_{W_j}(L)). 
    \end{equation} 
Here, $L\le \R^n$ means $L$ is a subspace of $\R^n$, and $\pi_{W_j}:\R^n\rightarrow W_j$ is the orthogonal projection. 

There exists $\nu>0$ depending on $\{W_j\}_{j=1}^J$, so that the following is true. For any $\e>0$ and for any $\cV_{j}=\{V_{j}\}\subset A(n-k,n)\ (j=1,\dots,J)$ being sets of $(n-k)$-planes, such that each $V_{j}\in \cV_{j}$ is orthogonal to some $W\in Q_\nu(W_j)\subset G(k,n)$, we have
\begin{equation}
    \int_{[-1,1]^n} \prod_{j=1}^J\bigg(\sum_{V_{j}\in\cV_{j}}1_{V_{j,\de}} \bigg)^{\frac{p}{J}}\lesssim \de^{n-\e}\de^{-\BL(\{W_j\}_{j=1}^J,p)}\prod_{j=1}^J\bigg(\#\cV_{j}\bigg)^{\frac{p}{J}}.
\end{equation}
\end{theorem}

By a compactness argument, we deduce the following result.

\begin{theorem}\label{BLthm}
    Let $\cW_j\subset G(k,n)$ be a compact set for $j=1,\dots,J$. Fix $p\in[1,\infty)$. 
Define
\begin{equation}\label{bl}
\BL(\{\cW_j\}_{j=1}^J,p):=\sup_{W_1\in\cW_1,\dots,W_J\in\cW_J }\BL(\{W_j\}_{j=1}^J,p). 
\end{equation}
Fix any $\e>0$.
Then for any $\cV_{j}=\{V_{j}\}\subset A(n-k,n)\ (j=1,\dots,J)$ being sets of $(n-k)$-planes, such that each $V_{j}\in \cV_{j}$ is orthogonal to some $W\in\cW_j$, we have
\begin{equation}\label{BL}
    \int_{[-1,1]^n} \prod_{j=1}^J\bigg(\sum_{V_{j}\in\cV_{j}}1_{V_{j,\de}} \bigg)^{\frac{p}{J}}\le C(\{\cW_j\},p,\e) \cdot \de^{n-\e}\de^{-\BL((\cW_j)_{j=1}^J,p)}\prod_{j=1}^J\bigg(\#\cV_{j}\bigg)^{\frac{p}{J}}.
\end{equation}
\end{theorem}

We briefly explain how to use compactness argument to deduce Theorem \ref{BLthm} from Theorem \ref{Domthm}. For any $\vec W=(W_1,\dots,W_J)\in \cW_1\times \dots\times \cW_J$, we use Theorem \ref{Domthm} to find a small number $\nu$, so that \eqref{BL} is true when all the slabs are in $\nu$-neighborhood of $W_j^\perp$. In other words, there exists a $O(\nu)$-neighborhood of $(W_1,\dots,W_J)$, for which we call $B(\vec W)$, so that \eqref{BL} is true when all the slabs are in $B(\vec W)$. Then we use the compactness of $\cW_1\times \dots\times \cW_J$ to find a finite covering using $B(\vec W)$. Using this finite covering and splitting the LHS of \eqref{BL} by triangle inequality, the proof of \eqref{BL} is not hard.

\bigskip
We continue the estimate of \eqref{ineq22}.
We recall \eqref{ineq3} here:
\begin{equation}
     \int_{\tau} \prod_{i=1}^d (\sum_{V\in E\cap U_i}\sum_{T\in\T_V, T\cap\tau\neq\emptyset}1_T)^{\frac{p}{d}}=K^{-(r-1)(n-k)}\int_{B^n(0,1)} \prod_{i=1}^d (\sum_{\wt V\in \cL^*(E\cap U_i)}\sum_{\wt T\in\wt\T_{\wt V,\tau}}1_{\wt T})^{\frac{p}{d}}.
\end{equation}
Note that for each $i$, $\cL^*(E\cap U_i)\subset \cL^*(U_i)=\wt U_i$ is an $O(K^{-1})$-ball. Applying Theorem \ref{BLthm} with $J=d$, $\cW_i=\wt U_i$, $\{V_{i,\de}\}=\bigcup_{\wt V\in \cL^*(E\cap U_i)}\T_{\wt V,\tau}$,
we obtain
\begin{equation}\label{ineq44}
\begin{split}
    \int_{B^n(0,1)} \prod_{i=1}^d (\sum_{\wt V\in \cL^*(E\cap U_i)}\sum_{\wt T\in\wt\T_{\wt V,\tau}}1_{\wt T})^{\frac{p}{d}} &\le  C_K \de^{n-\e^2}\de^{-\BL}\prod_{j=1}^d(\sum_{\wt V\in \cL^*(E\cap U_i)}\# \wt \T_{\wt V,\tau})^{\frac{p}{d}}\\
    &\le  C_K \de^{n-\e^2}\de^{-\BL}(\sum_{\wt V\in \cL^*(E)}\# \wt \T_{\wt V,\tau})^{p}
\end{split}
\end{equation}
Here,
\begin{equation}\label{defBL}
    \BL=\sup_{V_1\in \wt U_1,\dots,V_d\in\wt U_d} \sup_{L\le \R^n}(\dim L-\frac{p}{d} \sum_{j=1}^d\dim \pi_{V_j}(L) ). 
\end{equation} 

Combining \eqref{ineq0}, \eqref{ineq22}, \eqref{ineq3}, \eqref{ineq44}, and noting that we just bound the non-important factor $K^{-(r-1)(n-k)}$ in \eqref{ineq3} by $1$, we obtain
\begin{equation}\label{upperbound99}
\begin{split}
    \int_{B^n(0,1)} \prod_{i=1}^d (\sum_{V\in E\cap U_i}\sum_{T\in\T_V}1_T)^{\frac{p}{d}}&\le  C_K \de^{n-\e^2}\de^{-\BL} \sum_\tau (\sum_{\wt V\in \cL^*(E)}\# \wt \T_{\wt V,\tau})^p\\
    &\le  C_K \de^{n-\e^2}\de^{-\BL}  (\sum_\tau\sum_{\wt V\in \cL^*(E)}\# \wt \T_{\wt V,\tau})^p\\
    (\textup{by~} \eqref{sumTV})&\lesssim  C_K \de^{n-\e^2}\de^{-\BL} (\sum_{V\in E}\#\T_V)^p
\end{split}
\end{equation}

Next, we consider the quantity $\BL$. Fix 
$L\in G(l,n)$. Recalling the definition of $m(t,l)$ in Definition \ref{defm}, we have
\begin{equation}
    E(m(t,l)-1,l)\le t.
\end{equation}
Noting that $E(m,l)$ is given by \eqref{defE}. For $L\in G(l,n)$, we have
\begin{equation}
    \dim(\{ V\in G(k,n): \dim(\pi_L(V))\le m(t,l)-1 \})=E(m(t,l)-1,l)\le t.
\end{equation}
This means that 
\begin{equation}\label{submfd}
    \cM_L:=\{ V\in G(k,n): \dim(\pi_L(V))\le m(t,l)-1 \}
\end{equation} 
is at most a $t$-dimensional submanifold of $G(k,n)$. We also remark that for different $L\in G(l,n)$, the submanifolds \eqref{submfd} are isometric by symmetry. Noting that $\wt U_i$ ($1\le i\le d$) are disjoint $K^{-1}$-balls in $G(k,n)$, we have
\begin{equation}\label{boundofU}
    \#\{\wt U_i: 1\le i\le d, \wt U_i\cap \cM_L\neq \emptyset  \}\lesssim K^t\le K^{\tau-2\e}.
\end{equation}
The last inequality is from \eqref{taubigt}.
Returning to \eqref{defBL}, we see that for any $V_i\in \wt U_i$ $(1\le i\le d)$ and $L\in G(l,n)$, there are more than $d-CK^{\tau-2\e}$ many $\wt U_i$ such that $\dim(\pi_L(V_i))\ge m(t,l)$, where $C$ is a constant. Also recall that $d=\lfloor K^{\tau-\e} \rfloor$. As a result, we have
\begin{equation}
\begin{split}
    \BL&\le \sup_{l\in \Z\cap [0,n]}(l-\frac{p}{d} (d-CK^{\tau-2\e})m(t,l))\\
    &\le \sup_{l\in \Z\cap [0,n]}(l-p(1-2CK^{-\e})m(t,l))\\
    &\le \sup_{l\in \Z\cap [0,n]}(l-p\cdot m(t,l))+\e\\
    &=\lambda_p+\e,
\end{split}
\end{equation}
if $K$ is large enough.

Combining \eqref{left} and \eqref{upperbound99}, and noting that there are $\lesssim \de^{-\e}$ terms on the RHS of \eqref{left}, we obtain

\begin{equation}
   \int_{W}\Big(\sum_{V\in E} \sum_{T\in\T_{V}}1_T  \Big)^p\lesssim_{\e,p} \de^{-O(\e)}\de^n \de^{-\lambda_p}(\de^{-\si-\tau})^p.
\end{equation}
Renaming $O(\e)$ to be $\ga$ will give \eqref{discreteineq}.

\section{Lower bounds for \texorpdfstring{$T(a,s)$}{}}\label{section6}

In this section, we prove some lower bounds for $T(a,s)$ (see \eqref{defT}). For each pair $(a,s)$, we will construct an example $A\subset \R^n$ and calculate the dimension of the exceptional set
\[ \dim(\{V\in G(k,n):\dim(\pi_V(A))<s\}). \]
We remark that
in the setting of finite field, examples were built for some specific pairs $(a,s)$ in \cite{bright2023exceptional}.

\bigskip

We first recall the theorem of Ren and Wang \cite{ren2023furstenberg}:
\begin{theorem}\label{conj1}
Let $A\subset \R^2$ with $\dim(A)=a$. For $\theta\in G(1,\R^2)$, let $\pi_\theta: \R^2\rightarrow \theta$ be the orthogonal projection onto line $\theta$. For $0<s<\max\{1,a\}$, define $E_s(A):=\{\theta: \dim(\pi_\theta(A))<s\}$. Then
\begin{equation}\label{exestR2}
    \dim(E_s(A))\le \max\{0,2s-a\}.
\end{equation}
\end{theorem}

For our purpose, we are interested in the sharp example $A\subset \R^2$ for which \eqref{exestR2} holds with ``$\le$" replaced by ``$=$".
Actually, there exists $\bdA_{s,a}\subset \R^2$ with $\dim(\bdA_{s,a})=a$ such that $\bdE_{s,a}:=E_s(\bdA_{s,a})$ has Hausdorff dimension $\ge 2s-a$. For this example, see Section 5.4 in \cite{mattila2015fourier}. We would like to formally state it as a proposition.
\begin{proposition}\label{bdAE}
    For $0<a<2$, $\frac{a}{2}\le s<\min\{1,a\}$. There exists a set $\bdA_{s,a}\subset \R^2$ with $\dim(\bdA_{s,a})=a$ such that
    \begin{equation}
        \bdE_{s,a}:=\{\theta\in G(1,2): \dim(\pi_\theta(\bdA_{s,a}))<s\}
    \end{equation}
has Hausdorff dimension $2s-a$.
\end{proposition}

The sets $\bdA_{s,a}$ and $\bdE_{s,a}$ will serve as building blocks when we construct examples for all $n$.
We would like to spend some time explaining a construction of $\bdA_{s,a}$ and $\bdE_{s,a}$ in the finite setting. It is exactly the sharp example of the Szemer\'edi-Trotter theorem.

\begin{example}\label{ex1}
{\rm
Let $N$ be a large number. We assume $2s\ge a$.  Consider a set of lines
$\cL=\{l_{k,m}: |k|\le N^{2s-a}, |m|\le 10 N^s\}$ in $\R^2$. Here, $k$ and $m$ are integers and $l_{k,m}$ is given by
\[ l_{k.m}: y=kx+m. \]
We see that $\cL$ consists of lines from $\sim N^{2s-a}$ many directions, and in each of these directions there are $\sim N^s$ many lines. We denote these directions by $E$, and for $\theta\in E$, let $\cL_\theta$ be the lines in $\cL$ that are in direction $\theta$.

Consider the set $A:=\{(x,y)\in \Z^2: |x|\le N^{a-s}, |y|\le N^s\}$. For any $|k|\le N^{2s-a}$, we see that any $(x,y)\in A$ satisfies $|y-kx|\le 10N^s$. This means that for any direction $\theta\in E$, $A$ is covered by $\cL_\theta$. Therefore we have for each $\theta\in E$,
\[ \#\pi_\theta(A)\le \#\cL_\theta\lesssim N^s. \] We obtain the following estimate:
\[ \#\{\theta: \#\pi_\theta(A)\lesssim N^s\}\ge \#E\sim N^{2s-a}. \]
Such $A$ and $E$ are like the discrete analogue of $\bdA_{s,a}$ and $\bdE_{s,a}$.
}
\end{example}

\medskip

\subsection{Examples in \texorpdfstring{$\R^3$}{}}
To give more insight, we first focus on the projection to lines and planes in $\R^3$. In other words, we consider the case $n=3$, $k=1$ or $2$.

\begin{example}[Projections to lines in $\R^3$]\hfill
{\rm
\begin{enumerate}
    \item When $a\le 1$, let $A$ be contained in the $x_3$-axis. We see that for any line $\ell$ that is parallel to the $(x_1,x_2)$-plane, the projection of $A$ onto $\ell$ has dimension $0(<s)$. Therefore $\dim(E_s(A))\ge 1$.

    \item Consider the case $1<a\le 2$. 
    
    $\bullet$ When $\frac{a-1}{2}<s< a-1$, we choose $A=\bdA_{s,a-1}\times \R$. This looks like the left picture in Figure \ref{projectiontoplane}. We only look at those directions $\theta$ that are parallel to the $(x_1,x_2)$-plane. This is just the problem in $\R^2$. We see that $E_s(A)\supset \bE_{s,a-1}$, so
    \[ \dim(E_s(A))\ge \max\{0,2s-(a-1)\}=\max\{0,1+2s-a\}. \]

    $\bullet$ When $a-1<s<1$, we choose $A=A'\times \R$, where $A'\subset \R^2$ is any set with $\dim(A')=a-1$.
    We see that if $\theta$ is parallel to the $(x_1,x_2)$-plane, then $\pi_\theta(A)=\pi_\theta(A')$ has dimension less than $\dim(A')=a-1<s$. Therefore,
    \[\dim(E_s(A))\ge 1.\]
    
    \item Consider the case $2<a\le 3$. When $\frac{a-1}{2}<s<1$, we choose $A=\bdA_{s,a-1}\times \R$ (then $\dim(A)=\dim(\bdA_{s,a-1})+\dim(\R)=a$). We only look at those directions $\theta$ that are parallel to the $(x_1,x_2)$-plane. This is just the problem in $\R^2$. We see that $E_s(A)\supset \bdE_{s,a-1}$, so
    \[\dim(E_s(A))\ge \max\{0,2s-(a-1)\}=\max\{0,1+2s-a\}.\]
\end{enumerate}
}
\end{example}

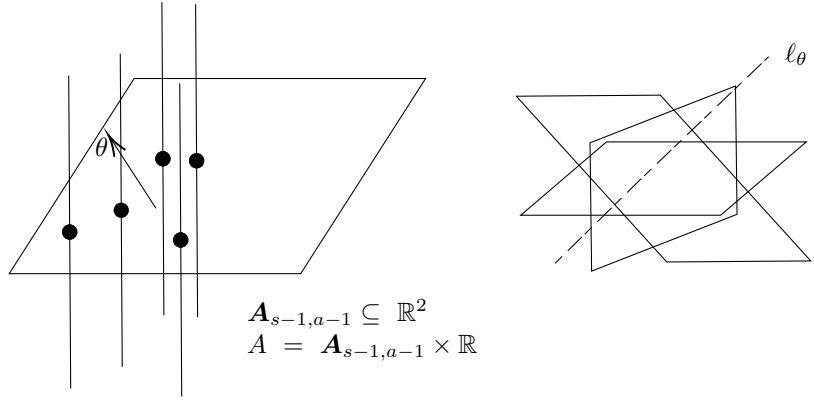
\begin{figure}
\begin{tikzpicture}[x=0.75pt,y=0.75pt,yscale=-1,xscale=1]
%uncomment if require: \path (0,300); %set diagram left start at 0, and has height of 300

%Shape: Parallelogram [id:dp0052352796552472824] 
\draw  [color={rgb, 255:red, 0; green, 0; blue, 0 }  ,draw opacity=1 ] (163,71.5) -- (310,71.5) -- (247,170) -- (100,170) -- cycle ;
%Straight Lines [id:da7942099118913704] 
\draw [color={rgb, 255:red, 0; green, 0; blue, 0 }  ,draw opacity=1 ]   (130,70) -- (131,228) ;
%Straight Lines [id:da34485573175682616] 
\draw [color={rgb, 255:red, 0; green, 0; blue, 0 }  ,draw opacity=1 ]   (156,59) -- (157,217) ;
%Straight Lines [id:da051332002117250086] 
\draw [color={rgb, 255:red, 0; green, 0; blue, 0 }  ,draw opacity=1 ]   (177,33) -- (178,191) ;
%Straight Lines [id:da871872591763849] 
\draw [color={rgb, 255:red, 0; green, 0; blue, 0 }  ,draw opacity=1 ]   (194,34) -- (195,192) ;
%Straight Lines [id:da3018150909975321] 
\draw [color={rgb, 255:red, 0; green, 0; blue, 0 }  ,draw opacity=1 ]   (186,74) -- (187,232) ;
%Straight Lines [id:da40826540298597647] 
\draw [color={rgb, 255:red, 0; green, 0; blue, 0 }  ,draw opacity=1 ] (174,137) -- (150,100);

% \draw [shift={(241,128)}, rotate = 172.35] [color={rgb, 255:red, 0; green, 0; blue, 0 }  ,draw opacity=1 ][line width=0.75]    (10.93,-3.29) .. controls (6.95,-1.4) and (3.31,-0.3) .. (0,0) .. controls (3.31,0.3) and (6.95,1.4) .. (10.93,3.29)   ;

\draw [shift={(150,100)}, rotate = 58] [color={rgb, 255:red, 0; green, 0; blue, 0 }  ,draw opacity=1 ][line width=0.75]    (10.93,-3.29) .. controls (6.95,-1.4) and (3.31,-0.3) .. (0,0) .. controls (3.31,0.3) and (6.95,1.4) .. (10.93,3.29)   ;

%Straight Lines [id:da43147709161021375] 
\draw [color={rgb, 255:red, 0; green, 0; blue, 0 }  ,draw opacity=1 ]   (130.5,149) ;
\draw [shift={(130.5,149)}, rotate = 0] [color={rgb, 255:red, 0; green, 0; blue, 0 }  ,draw opacity=1 ][fill={rgb, 255:red, 0; green, 0; blue, 0 }  ,fill opacity=1 ][line width=0.75]      (0, 0) circle [x radius= 3.35, y radius= 3.35]   ;
\draw [shift={(130.5,149)}, rotate = 0] [color={rgb, 255:red, 0; green, 0; blue, 0 }  ,draw opacity=1 ][fill={rgb, 255:red, 0; green, 0; blue, 0 }  ,fill opacity=1 ][line width=0.75]      (0, 0) circle [x radius= 3.35, y radius= 3.35]   ;
%Straight Lines [id:da5567651658136805] 
\draw [color={rgb, 255:red, 0; green, 0; blue, 0 }  ,draw opacity=1 ]   (156.5,138) ;
\draw [shift={(156.5,138)}, rotate = 0] [color={rgb, 255:red, 0; green, 0; blue, 0 }  ,draw opacity=1 ][fill={rgb, 255:red, 0; green, 0; blue, 0 }  ,fill opacity=1 ][line width=0.75]      (0, 0) circle [x radius= 3.35, y radius= 3.35]   ;
 
\draw [color={rgb, 255:red, 0; green, 0; blue, 0 }  ,draw opacity=1 ]   (177.5,112) ;
\draw [shift={(177.5,112)}, rotate = 0] [color={rgb, 255:red, 0; green, 0; blue, 0 }  ,draw opacity=1 ][fill={rgb, 255:red, 0; green, 0; blue, 0 }  ,fill opacity=1 ][line width=0.75]      (0, 0) circle [x radius= 3.35, y radius= 3.35]   ;
\draw [shift={(177.5,112)}, rotate = 0] [color={rgb, 255:red, 0; green, 0; blue, 0 }  ,draw opacity=1 ][fill={rgb, 255:red, 0; green, 0; blue, 0 }  ,fill opacity=1 ][line width=0.75]      (0, 0) circle [x radius= 3.35, y radius= 3.35]   ;
%Straight Lines [id:da7466854730127388] 
\draw [color={rgb, 255:red, 0; green, 0; blue, 0 }  ,draw opacity=1 ]   (194.5,113) ;
\draw [shift={(194.5,113)}, rotate = 0] [color={rgb, 255:red, 0; green, 0; blue, 0 }  ,draw opacity=1 ][fill={rgb, 255:red, 0; green, 0; blue, 0 }  ,fill opacity=1 ][line width=0.75]      (0, 0) circle [x radius= 3.35, y radius= 3.35]   ;
\draw [shift={(194.5,113)}, rotate = 0] [color={rgb, 255:red, 0; green, 0; blue, 0 }  ,draw opacity=1 ][fill={rgb, 255:red, 0; green, 0; blue, 0 }  ,fill opacity=1 ][line width=0.75]      (0, 0) circle [x radius= 3.35, y radius= 3.35]   ;
%Straight Lines [id:da15009428697497573] 
\draw [color={rgb, 255:red, 0; green, 0; blue, 0 }  ,draw opacity=1 ]   (186.5,153) ;
\draw [shift={(186.5,153)}, rotate = 0] [color={rgb, 255:red, 0; green, 0; blue, 0 }  ,draw opacity=1 ][fill={rgb, 255:red, 0; green, 0; blue, 0 }  ,fill opacity=1 ][line width=0.75]      (0, 0) circle [x radius= 3.35, y radius= 3.35]   ;
\draw [shift={(186.5,153)}, rotate = 0] [color={rgb, 255:red, 0; green, 0; blue, 0 }  ,draw opacity=1 ][fill={rgb, 255:red, 0; green, 0; blue, 0 }  ,fill opacity=1 ][line width=0.75]      (0, 0) circle [x radius= 3.35, y radius= 3.35]   ;
%Shape: Parallelogram [id:dp43024367487247206] 
\draw   (392.93,103.95) -- (393.72,168.71) -- (467.07,140.05) -- (466.28,75.29) -- cycle ;
%Shape: Parallelogram [id:dp22581504330104285] 
\draw   (458.8,140.5) -- (358,140.5) -- (401.2,103.5) -- (502,103.5) -- cycle ;
%Shape: Parallelogram [id:dp01826001029486002] 
\draw   (431.61,164.1) -- (355.74,80.43) -- (428.39,79.9) -- (504.26,163.57) -- cycle ;
%Straight Lines [id:da9103107403593507] 
\draw  [dash pattern={on 3.75pt off 3pt on 7.5pt off 1.5pt}]  (375.28,164.79) -- (418.12,122.66) -- (483,60.5) ;

% Text Node
\draw (212,181.4) node [anchor=north west][inner sep=0.75pt]    {$ \begin{array}{l}
\bdA_{s-1,a-1}\subseteq \ \mathbb{R}^{2}\\
A\ =\ \bdA_{s-1,a-1}\times \mathbb{R}
\end{array}$};
% Text Node
\draw (142,100) node [anchor=north west][inner sep=0.75pt]    {$\theta $};
% Text Node
\draw (490,52.4) node [anchor=north west][inner sep=0.75pt]    {$\ell_\theta$};
\end{tikzpicture}

\caption{Projection to planes}
\label{projectiontoplane}
\end{figure}

\begin{example}[Projections to planes in $\R^3$]\hfill
{\rm
    \begin{enumerate}
    \item When $a\le 1$, let $A=\bdA_{s,a}\times \{0\}$. For $\theta\in S^2$, we use $V_\theta$ to denote the plane orthogonal to $\theta$. 
    We only look at the projection of $A$ to those $V_\theta$ where $\theta$ is parallel to the $(x_1,x_2)$-plane. This exactly becomes a projection-to-line problem in $\R^2$, so
    we have that $\dim(E_s(A))\ge \dim(\bdE_{s,a})\ge \max\{0,2s-a\}$.

    \item Consider the case $1<a\le 2$. 
    
    $\bullet$ When $s\le 1$, we use the same example as in the last paragraph.

    $\bullet$ When $1<s\le \frac{a+1}{2}$, we choose $A$ to be a subset of $\R^2\times \{0\}$. Then for those $\theta$ parallel to $(x_1,x_2)$-axis, the projection of $A$ onto $V_\theta$ is contained in the line $(\R^2\times \{0\})\cap V_\theta$ which has dimension $\le 1<s$. This means that these $V_\theta$ are in the exceptional set. So, $\dim(E_s(A))\ge 1$.

    $\bullet$ When $\frac{a+1}{2}<s<a$, we choose $A=\bdA_{s-1,a-1}\times \R$ (see Figure \ref{projectiontoplane}). For a direction $\theta$ that is parallel to the $(x_1,x_2)$-plane, we use $\ell_\theta$ to denote the line contained in $\R^2\times \{0\}$ that is orthogonal to $\theta$, i.e., $\ell_\theta=V_\theta\cap (\R^2\times \{0\})$. Suppose $\ell_\theta$ is in $\bdE_{s-1,a-1}$, i.e., $\dim(\pi_{\ell_\theta}(\bdA_{s-1,a-1}))<s-1$. Then noting that $\pi_{V_\theta}(\bdA_{s-1,a-1}\times \R)=\pi_{\ell_\theta}(\bdA_{s-1,a-1})\times \R$, we have $\dim(\pi_{V_\theta}(A))=\dim(\pi_{\ell_\theta}(\bdA_{s-1,a-1}))+1<s$. On the right hand side of Figure \ref{projectiontoplane}, we draw a bunch of planes that contain $\ell_\theta$ but not contain the $x_3$-axis. We can view these planes as 
    \begin{equation}\label{view}
        \R\P^1\setminus \{*\},
    \end{equation} where $*$ corresponds to the $(x_1,x_2)$-plane. The geometric observation is that the projection of $A$ onto these planes are the same, and hence has dimension $<s$. We see that 
    \begin{align*}
        \dim(E_s(A))&\ge \dim(\bdE_{s-1,a-1}\times (\R\P^1\setminus\{*\}))\\
        &=\dim(\bdE_{s-1,a-1})+\dim (\R\P^1\setminus\{*\})\\
        &\ge 2(s-1)-(a-1)+1=2s-a.
    \end{align*}
    We used that $\dim(\bdE_{s-1,a-1})\ge\max\{0,2(s-1)-(a-1)\}=2(s-1)-(a-1)$ since $\frac{a+1}{2}<s$.

    \item Consider the case $2<a\le 3$.
    
    $\bullet$ When $a-1<s<\frac{a+1}{2}$, we choose $A$ to be contained in $\R^2\times I$, where $I\subset \R$ has dimension $a-2$. Then for those $\theta$ parallel to $(x_1,x_2)$-axis, the projection of $A$ onto $V_\theta$ is contained in $\ell_\theta\times I$ which has dimension $\le 1+a-2=a-1<s$. This means that these $V_\theta$ are in the exceptional set. So, $\dim(E_s(A))\ge 1$. 
    
    $\bullet$ When $\frac{a+1}{2}<s<2$, we just use the same example $A=\bdA_{s-1,a-1}\times \R$ as above to get the lower bound $\dim(E_s(A))\ge 2s-a$.
\end{enumerate}
}
\end{example}

We summarize our examples as follows.

\begin{proposition}[Projections to lines in $\R^3$]\label{prop1}
Consider the projection to lines in $\R^3$, so $E_s(A)$ is given by \eqref{exset} with $k=1, n=3$. We have
\begin{enumerate}
    \item When $a\le 1:$ $T(a,s)\ge 1$;
    \item When $1<a\le 2:$ 
    \begin{equation*}
        T(a,s)\ge\begin{cases}
            0 & s\le \frac{a-1}{2}\\
            1+2s-a & \frac{a-1}{2}<s\le a-1\\
            1 & a-1<s<1
        \end{cases}
    \end{equation*}
    \item When $2<a\le 3:$ 
    \begin{equation*}
        T(a,s)\ge\begin{cases}
            0 & s\le \frac{a-1}{2}\\
            1+2s-a & \frac{a-1}{2}<s<1
        \end{cases}
    \end{equation*}
\end{enumerate}
\end{proposition}

\begin{proposition}[Projections to planes in $\R^3$]\label{prop2}
Consider the projection to planes in $\R^3$, so $E_s(A)$ is given by \eqref{exset} with $k=2, n=3$. We have
\begin{enumerate}
    \item When $a\le 1:$ $T(a,s)\ge \max\{0,2s-a\}$;
    \item When $1<a\le 2:$ 
    \begin{equation*}
        T(a,s)\ge\begin{cases}
            0 & s\le \frac{a}{2}\\
            2s-a & \frac{a}{2}<s\le 1\\
            1 & 1<s\le\frac{a+1}{2}\\
            2s-a & \frac{a+1}{2}<s<a
        \end{cases}
    \end{equation*}
    \item When $2<a\le 3:$ 
    \begin{equation*}
        T(a,s)\ge\begin{cases}
            0 & s\le a-1\\
            1 & a-1<s\le \frac{a+1}{2}\\
            2s-a & \frac{a+1}{2}<s<2
        \end{cases}
    \end{equation*}
\end{enumerate}
\end{proposition}

\begin{remark}\label{rm5}
    {\rm
    To construct analogous examples in higher dimensions, we can play the same trick. For example, choose $A=\bdA_{s',a'}\times \R^l$, or $I\times \R^{l'}$. Then we may find the lower bound of $T(a,s)$ are of forms $m+2s-a$, or $m'$, for some integers $m,m'$. We will talk about the examples in higher dimensions in the next subsection.
    }
\end{remark}

\subsection{Examples in higher dimensions}
We discuss the examples of the exceptional sets in higher dimensions. There will be four types of the examples. We will use $d(m,n)$ to denote $\dim(G(m,n))$, namely $d(m,n)=m(n-m)$. We first list some well known facts.

\begin{lemma}\label{easylem}
Suppose $V,W\subset \R^n$ are two subspaces. Then 
\[\dim(\pi_V(W))=\dim(\pi_W(V)),\] 
and \[\dim(\pi_V(W))=\dim(W)-\dim(W\cap V^\perp).\]
\end{lemma}
\begin{proof}
Since the linear map $\pi_V: W\rightarrow V$ has kernel $W\cap V^\perp$, we have
\[\dim(\pi_V(W))=\dim(W)-\dim(W\cap V^\perp).\]
To prove the first equality, we note
\begin{align*}
  \dim(\pi_V(W))&=\dim(W)-\dim(W\cap V^\perp)\\
  &=\dim W -\bigg(n-\dim\big((W\cap V^\perp)^\perp\big)\bigg)\\
  &=\dim W-n +\dim(W^\perp+V)\\
  &=\dim W-n+ \dim W^\perp+ \dim V-\dim(W^\perp\cap V)\\
  &=\dim V-\dim(V\cap W^\perp)\\
  &=\dim(\pi_W(V)).
\end{align*}
\end{proof}

\begin{lemma}\label{easylem2}
Let $0\le l\le n$, $1\le k,m\le n$ be integers satisfying $n-k\ge m-l$ and $l\le k,m$. If $W\in G(m,n)$, then
\[ \dim\{ V\in G(k,n):\dim(\pi_V(W))\le l \}=d(k-l,n-m)+d(l,n-(k-l)). \]
\end{lemma}
\begin{proof}
By Lemma \ref{easylem}, we have
\begin{align*}
    \{ V\in G(k,n):\dim(\pi_V(W))\le l \}&=\{V\in G(k,n): \dim(V\cap W^\perp)\ge k-l\},
\end{align*} 
By condition $l\le k$, we see $\dim(V\cap W^\perp)\ge k-l$ makes sense.
    Actually, $\{ V\in G(k,n):\dim(V\cap W^\perp)\ge k-l \}$ is a Schubert cycle whose dimension is not hard to calculate (see \cite{griffiths2014principles} Chapter 1, Section 5). Here, we provide a sketch of proof. The condition $n-k\ge m-l$ says that $\dim(W^\perp)\ge k-l$. The condition $l\le m$ yields that generic $V\in G(k,n)$ satisfies $\dim(V\cap W^\perp)=\dim(V)+\dim(W^\perp)-n=k+n-m-n \le k-l$ (imagine the generic case is when $V$ and $W^\perp$ are as transverse as possible). Therefore, generic $V\in \{ V\in G(k,n):\dim(V\cap W^\perp)\ge k-l \}$ will satisfy $\dim(V\cap W^\perp)=k-l$. It is equivalent to calculate the dimension of 
    \[\{V\in G(k,n): \dim(V\cap W^\perp)=k-l\}.\]
Note that we can identify $V$ with $(v_1,\dots,v_{k})$, where $\{v_i\}_{i=1}^{k}$ are orthonormal basis of $V$. We first choose $v_1,\dots,v_{k-l}\in W^\perp$, which is equivalent to choosing an $(k-l)$-plane in $W^\perp$. It gives $d(k-l,n-m)$-dimensional choices. Next, we choose $v_{k-l+1},\dots,v_k$ from $\textup{span}\{v_1,\dots,v_{k-l}\}^\perp$, which gives $d(l,n-(k-l))$-dimensional choices.
Adding together, gives the dimension $d(k-l,n-m)+d(l,n-(k-l))$.    
\end{proof}

We state the key result.

\medskip

\begin{proposition}\label{highprop}
    Fix $0<k<n$. For $A\subset \R^n$, define
    \[ E_s(A):=\{V\in G(k,n): \dim(\pi_V(A))<s\}. \]
For $0<a<n$ and $0<s\le \min\{k,a\}$, define
\[ T(a,s):=\sup_{\dim(A)=a}\dim(E_s(A)). \]
Then we have the following lower bound for $T(a,s)$. We write $a=m+\beta, s=l+\ga$ where $m,l\in\N$ and $\beta,\ga\in(0,1]$.

\begin{enumerate}
\item \fbox{{\bf Type 1:} $1\le m\le n+l-k, \ga>\frac{\beta+1}{2}$} 
\begin{align}\label{t1}
    T(m+\beta,l+\ga)&\ge d(k-l,n-m)+d(l,n-(k-l))+2\ga-(\beta+1)\\
    \nonumber&=k(n-k)-(m-l)(k-l)+2\ga-(\beta+1).
\end{align}

\item \fbox{{\bf Type 2:} $m\le n+l-k, \ga>\beta$} 
\begin{align}\label{t2}
    T(m+\beta,l+\ga)&\ge d(k-l,n-m)+d(l,n-(k-l))\\
    \nonumber&=k(n-k)-(m-l)(k-l).
\end{align}

\item \fbox{{\bf Type 3:} $m\le n+l-k-1, \ga>\frac{\beta}{2}$} 
\begin{align}\label{t3}
T(m+\beta,l+\ga)&\ge d(k-l,n-m-1)+d(l,n-(k-l))+2\ga-\beta\\
\nonumber&=k(n-k)-(m+1-l)(k-l)+2\ga-\beta.
\end{align}

\item \fbox{{\bf Type 4:} $1\le l, m\le n+l-k-1$} 
\begin{align}\label{t4}
    T(m+\beta,l+\ga)&\ge d(k-(l-1),n-m)+d(l-1,n-(k-l+1))\\
    \nonumber&=k(n-k)-(m-l+1)(k-l+1).
\end{align}

\end{enumerate}

\end{proposition}

\begin{proof}
    \noindent
{\bf Type 2}: When $n-m\ge k-l, \ga> \beta$, then \begin{equation}\label{typ2}
    T(m+\beta,l+\ga)\ge d(k-l,n-m)+d(l,n-(k-l))=k(n-k)-(m-l)(k-l).
\end{equation}

We choose $A=\R^m\times I\times \{\underbrace{(0,\dots,0)}_{(n-m-1)\times 0}\}$, where $I\subset \R$ has Hausdorff dimension $\beta$. For simplicity, we just write $A=\R^m\times I$. We want to find those $V\in G(k,n)$ such that $\dim(\pi_V(A))<l+\ga$. We note that if $V$ satisfies $\dim(\pi_V(\R^m))\le l$, then $\dim(\pi_V(\R^m\times I))\le l+\beta<l+\ga$, which means $V\in E_s(A)$. This shows that
\[ E_s(A)\supset \{V\in G(k,n): \dim(\pi_V(\R^m))\le l\}. \]
By Lemma \ref{easylem2}, the right hand side  has dimension $d(k-l,n-m)+d(l,n-(k-l))$. This gives \eqref{typ2}.

\medskip

\noindent
{\bf Type 3}: When $n-m-1\ge k-l, \ga> \frac{\beta}{2}$, then
\begin{align}\label{typ3}
T(m+\beta,l+\ga)&\ge d(k-l,n-m-1)+d(l,n-(k-l))+2\ga-\beta\\
\nonumber&=k(n-k)-(m+1-l)(k-l)+2\ga-\beta.
\end{align}

We choose $A=\bdA_{\ga,\beta}\times \R^m\times \{\underbrace{(0,\dots,0)}_{(n-m-2)\times 0}\}$. Recall that $\bdA_{\ga,\beta}\subset \R^2$ is the example discussed in Proposition
\ref{bdAE}. We will use $\R^2$ to denote $\R^2\times \{\underbrace{(0,\dots,0)}_{(n-2)\times 0}\}$. We use $\R^m$ to denote $\{(0,0)\}\times \R^m\times \{(0,\dots,0)\}$. For each direction $\theta$ parallel to $\R^2$, we use $\ell_\theta$ to denote the line parallel to $\R^2$ but orthogonal to $\theta$. We define
\[ \cU_\theta=\{ V\in G(k,n): \dim(\pi_V(\R\theta \oplus \R^m))\le l, \dim(\pi_V(\R^m))=l, \dim(\pi_V(\R^2\times \R^m))=l+1 \}, \]
\[\cV_\theta=\{ V\in G(k,n): \dim(\pi_V(\R\theta \oplus \R^m))\le l\},\]
\[\cW_\theta=\{V\in \cV_\theta: \dim(\pi_V(\R^m))\le l-1\}\cup \{V\in \cV_\theta: \dim(\pi_V(\R^2\times \R^m)\le l \}.\]
It is not hard to see $\cU_\theta=\cV_\theta\setminus \cW_\theta$. To get some intuition on these sets, we suggest readers to look back at \eqref{view}. $\cU_\theta, \cV_\theta$ and $\cW_\theta$ are higher dimensional analog of $\R\P^1\setminus \{*\}, \R\P^1$ and $ \{*\}$.

Let $\bE_{\ga,\beta}=\{ \theta\in S^1: \dim(\pi_{l_\theta}(\bdA_{\ga,\beta}))<\ga \}$. We claim that if $\theta\in \bE_{\ga,\beta}$, then $\cV_\theta\subset E_{l+\ga}(A)$. In other words, any $V\in \cV_\theta$ satisfies $\dim(\pi_V(A))<l+\ga$. We first note that $\bdA_{\ga,\beta}\subset \pi_{l_\theta}(\bdA_{\ga,\beta})+\R\theta$, so
\[ A=\bdA_{\ga,\beta}\times \R^m\subset (\pi_{l_\theta}(\bdA_{\ga,\beta})+\R\theta)\times \R^m= \pi_{l_\theta}(\bdA_{\ga,\beta})+(\R\theta\oplus\R^m). \]
Therefore,
\[ \dim(\pi_V(A))\le \dim(\pi_{l_\theta}(\bdA_{\ga,\beta}))+\dim(\pi_V(\R\theta\oplus\R^m))< \ga+l. \]

We see that $E_{l+\ga}(A)\supset \bigcup_{\theta\in \bE_{\ga,\beta}} \cV_\theta.$
$\cV_\theta$ may intersect with each other, so we delete a tiny portion $\cW_\theta$ from $\cV_\theta$, so that the remaining part $\cU_\theta=\cV_\theta\setminus \cW_\theta$ are disjoint with each other. Then we have
\[ E_{l+\ga}(A)\supset \bigsqcup_{\theta\in\bE_{\ga,\beta}}\cU_\theta. \]
To make the argument precise, we first apply Lemma \ref{easylem2} to see \[\dim(\cV_\theta)=d(k-l,n-m-1)+d(l,n-(k-l))=k(n-k)-(m+1-l)(k-l).\]
To calculate $\dim(\cW_\theta)$, we apply Lemma \ref{easylem2} again to see
\[ \dim(\{V\in G(k,n):\dim(\pi_V(\R^m))\le l-1\})=k(n-k)-(m-l+1)(k-l+1), \]
and
\[ \dim(\{V\in G(k,n):\dim(\pi_V(\R^2\times \R^m))\le l\})=k(n-k)-(m+2-l)(k-l). \]
Therefore, $\dim(\cW_\theta)<\dim(\cV_\theta)$ and hence
\[ \dim(\cU_\theta)=\dim(\cV_\theta)=d(k-l,n-m-1)+d(l,n-(k-l)). \]

Next we show $\cU_\theta\cap \cU_{\theta'}=\emptyset$ for $\theta\neq\theta'$.
By contradiction, suppose $V\in \cU_\theta\cap \cU_{\theta'}$. By definition, $V$ satisfies
\begin{align*}
\dim(\pi_V(\R\theta \oplus \R^m))\le l,\\
\dim(\pi_V(\R\theta' \oplus \R^m))\le l,\\
\dim(\pi_V(\R^m))=l,\\
\dim(\pi_V(\R^2\times \R^m))=l+1. 
\end{align*}
The first and third conditions imply $\pi_V(\R\theta)\subset \pi_V(\R^m)$. Similarly, the second and third conditions imply $\pi_V(\R\theta')\subset \pi_V(\R^m)$. Noting that $\textup{span}\{\R\theta,\R\theta'\}=\R^2$, we have 
\[ \pi_V(\R^2\times \R^m)\subset \pi_V(\R^m), \]
which has dimension $=l$, contradicting the fourth condition.

We have shown that $E_{l+\ga}(A)\supset \bigsqcup_{\theta\in \bE_{\ga,\beta}}\cU_\theta$, which has dimension
$\dim(\bE_{\ga,\beta})+\dim(\cU_\theta)$
\[ =2\ga-\beta+d(k-l,n-m-1)+d(l,n-(k-l)). \]
This gives \eqref{typ3}.

\medskip

\noindent
{\bf Type 1 }: When $m\ge 1, n-m\ge k-l, \ga>\frac{\beta+1}{2}$, then
\begin{align}\label{typ1}
    T(m+\beta,l+\ga)&\ge d(k-l,n-m)+d(l,n-(k-l))+2\ga-(\beta+1)\\
    &=k(n-k)-(m-l)(k-l)+2\ga-(\beta+1).
\end{align}

The trick is to write $T(m+\beta,l+\ga)=T((m-1)+(\beta+1),l+\ga)$ and then apply \textit{Type} 2 construction. We will choose 
\[ A=\bdA_{\ga,\beta+1}\times \R^{m-1}\times \{(0,\dots,0)\}. \]
It will give the bound \eqref{typ1}.

\medskip

\noindent
{\bf Type 4 }: When $l\ge 1, n-m\ge k-l+1$, then
\begin{align}\label{typ4}
    T(m+\beta,l+\ga)&\ge d(k-(l-1),n-m)+d(l-1,n-(k-l+1))\\
    \nonumber&=k(n-k)-(m-l+1)(k-l+1).
\end{align}

The idea is to apply \textit{Type} 1 construction. Note that here we do not have the condition $\ga>\beta$, so we cannot directly apply \textit{Type} 1. The trick is to write 
\[ T(m+\beta,l+\ga)\ge T(m+\beta,(l-1)+1). \]
Now our new $l$ becomes $l-1$, and our new $\ga$ becomes $1$. We apply \textbf{Type 1}, and obtain \eqref{typ4}.
\end{proof}

\bigskip

Given $a=m+\beta, s=l+\ga$, we compare the lower bounds of these four types and see that 
\[ {\bf Type\ 1}\ge {\bf Type\ 2}\ge {\bf Type\ 3}\ge {\bf Type\ 4}. \]
Therefore, the strategy to find the lower bound of $T(a,s)$ is by testing the requirement of {\bf Type 1}, {\bf Type 2},... one by one, and use the lower bound as soon as the requirement is satisfied. If neither requirement is satisfied, then we only expect the trivial lower bound $T(a,s)\ge 0$.  

For our application in Theorem \ref{explicitT}, we only use {\bf Type 2} example.

\bigskip

\subsection{Proof of Theorem \ref{explicitT}}
We will use {\bf Type 2} bound in Proposition \ref{highprop}. We consider the case $k\le \frac{n}{2}$. Suppose $\beta\in(0,1], \ga\in (\beta,\frac{k}{n}(1+\beta)]$. By Theorem \ref{thm}, we have
\begin{equation}
    T(1+\beta,\ga)\le T(1+\beta,\frac{k}{n}(1+\beta))\le k(n-k)-k.
\end{equation}
To obtain a lower bound, we use {\bf Type 2}. We choose $m=1,l=0$, so $a=1+\beta$ and $s=\ga$. We can verify the condition that $m\le n+l-k, \ga>\beta$. Therefore, we obtain
\begin{equation}
    T(1+\beta,\ga)\ge k(n-k)-k.
\end{equation}
This finishes the proof that \begin{equation}
    T(1+\beta,\ga)=k(n-k)-k.
\end{equation}

Next, we consider the case $k\ge \frac{n}{2}$. Suppose $\beta\in(0,1], \ga\in (\beta,(1-\frac{k}{n})+\frac{k}{n}\beta]$. By Theorem \ref{thm}, we have
\begin{equation}
    T(n-1+\beta,k-1+\ga)\le T(n-1+\beta,\frac{k}{n}(n-1+\beta))\le k(n-k)-(n-k).
\end{equation}
To obtain a lower bound, we use {\bf Type 2}. We choose $m=n-1,l=k-1$, so $a=n-1+\beta$ and $s=k-1+\ga$. We can verify the condition that $m\le n+l-k, \ga>\beta$. Therefore, we obtain
\begin{equation}
    T(n-1+\beta,k-1+\ga)\ge k(n-k)-(n-k).
\end{equation}
This finishes the proof that \begin{equation}
    T(n-1+\beta,k-1+\ga)= k(n-k)-(n-k).
\end{equation}

\bibliographystyle{abbrv}
\bibliography{refs}
%***************************************

\end{document}